\documentclass[10pt]{article}
\usepackage{fancyhdr}
\usepackage{amsmath, amsthm, hyperref, enumerate}
\usepackage[all, cmtip]{xy}
\usepackage{amssymb, amsfonts}
\usepackage[affil-it]{authblk}
\usepackage{baskervald}
\usepackage[T1]{fontenc}

\newcommand{\Jac}{\mathrm{Jac}}

\newcommand{\SL}{\mathrm{SL}}
\newcommand{\GL}{\mathrm{GL}}

\newcommand{\Sh}{\mathrm{Sh}}

\newcommand{\End}{\mathrm{End}}

\begin{document}
\newtheorem{prop}{Proposition}[section]
\newtheorem{lem}[prop]{Lemma}
\newtheorem{defn}[prop]{Definition}
\newtheorem{cor}[prop]{Corollary}
\newtheorem{thm}[prop]{Theorem} 
\newtheorem{thmx}{Theorem}
\renewcommand{\thethmx}{\Alph{thmx}}
\theoremstyle{remark} 
\newtheorem{example}{Example}[section]
\newtheorem{remark}{Remark}[section]

\title{Modularity of some elliptic curves over totally real fields.}
\author{Bao V. Le Hung}

\maketitle

\begin{abstract}
  We investigate modularity of elliptic curves over a general totally real number field, establishing a finiteness result for the set non-modular $j$-invariants. By analyzing quadratic points on some modular curves, we show that all elliptic curves over certain real quadratic fields are modular.
\end{abstract}
\tableofcontents
\section{Introduction}
The classical Shimura-Taniyama conjecture \cite{diamond_first_2005} is the statement that every elliptic curve $E$ over $\mathbb{Q}$ is associated to a cuspidal Hecke newform $f$ of the group $\Gamma_0(N)\subset\textrm{SL}_2(\mathbb{Z})$. Here the meaning of ''associated'' is that there is an isomorphism between compatible systems of $l$-adic representations of $G_\mathbb{Q}$ 
\begin{align*}
  \rho_{E,l}\simeq \rho_{f,l}
\end{align*}
where the left-hand side is the representation on the $l$-adic Tate module of $E$ and he right-hand side is the $l$-adic representation attached to $f$, or rather the corresponding cuspidal automorphic representation $\pi_f$ of $\GL_2(\mathbb{A}_\mathbb{Q})$ constructed by Eichler-Shimura. In the pioneering work \cite{wiles_modular_1995}, \cite{taylor_ring-theoretic_1995}, Wiles and Taylor-Wiles established the conjecture for all semi-stable $E$, which forms the heart of Wiles'proof of Fermat's Last theorem. After many gradual improvements \cite{diamond_taylor-wiles_1997}, \cite{conrad_modularity_1999}, the full conjecture is finally proven in \cite{breuil_modularity_2001}. It is then natural to try to study the generalization of the conjecture to more general number fields $F$, that is to show that all elliptic curves over $F$ have compatible systems of $l$-adic representations of $G_F$ associated to a cuspidal automorphic representation $\pi$ of $\GL_2(\mathbb{A}_F)$. Unfortunately, in this generality the existence of Galois representations associated to $\pi$ is not known. However, when $F$ is totally real, the required Galois representations have been constructed for some time by Carayol, Wiles, Blasius-Rogawski and Taylor \cite{carayol_sur_1983},\cite{wiles_ordinary_1988}, \cite{blasius_motives_1993} \cite{taylor_galois_1989}, \cite{taylor_galois_1995}, while when $F$ is CM, the Galois representations have only been constructed very recently \cite{HLTT}, \cite{scholze_torsion_2013}. In this paper, we focus our attention on the case $F$ totally real. Previous results in this direction include \cite{jarvis_fermat_2004}, \cite{jarvis_modularity_2008}, \cite{dieulefait_fermat-type_2011}, \cite{allen_modularity_2013}, establishing modularity under local restrictions on the elliptic curves or over particular fields. On the other hand, in the contemporary work \cite{freitas_modularity_2013} the authors establish modularity for elliptic curves over real quadratic fields with full 2-torsion, as well as a finiteness statement regarding possible non-modular elliptic curves with full 2-torsion over a general totally real field. In this work, we establish the following:
\begin{thmx}(see Theorem \ref{Proof of Theorem A})
Let $F$ be a fixed totally real number field. Then, up to isomorphism over $\overline{F}$,there are only finitely many elliptic curves $E$ defined over a totally real extension $F'/F$ of degree at most 2 that are not modular.
\end{thmx}
An immediate consequence is that there are only finitely many isomorphism classes (over $\overline
{\mathbb{Q}}$) of elliptic curves $E$ over a(n unspecified) real quadratic field such that $E$ is not modular. Under some restrictions on the field, we can show that in fact no such exceptions exists:
\begin{thmx} \label{Theorem B} (see Theorem \ref{Proof of Theorem B})
  If $F$ is real quadratic such that 5 and 7 are unramified in $F$, then any elliptic curve $E$ defined over $F$ are modular.
\end{thmx}
In fact, with the methods in this paper, it suffices to assume 5 is unramified, see remark \ref{General quadratic fields}. In forthcoming joint work with N.Freitas and S.Siksek, we will show that all elliptic curves over real quadratic fields are modular \cite{FLS}.

The proof of the above theorems follows the framework introduced by Wiles in \cite{wiles_modular_1995}. To prove $E$ is modular, it suffices to show any particular $\rho_{E,l}$ is modular. Modularity is then established in three steps:
\begin{itemize}
  \item Automorphy lifting: If $\rho_{E,l}$ is congruent mod $l$ to an automorphic $l$-adic representation then $\rho_{E,l}$ is also automorphic, under suitable hypotheses.
  \item 
  Establishing residual automorphy (''Serre's conjecture''): $\overline{\rho}_{E,l}$ is is automorphic for some prime $l$, and such that the previous step applies.
  \item 
  Understanding which elliptic curves can not be accessed by the previous two steps, and (ideally) establish automorphy for them by other means.
\end{itemize}
For the first step, the technology for automorphy lifting has improved greatly since \cite{wiles_modular_1995}, in our context the most important improvements are in \cite{kisin_moduli_2009}. This supplies very strong lifting statements by combining existing statements in the literature, under usual largeness assumptions of the residual image. We note however that automorphy lifting for small residual images in the literature remain too restrictive for our needs, see remark \ref{Automatic ordinarity}. What we need about automorphy lifting is summarized in section \ref{Mod lift}.

For the second step, we follow prime switching arguments of Wiles \cite{wiles_modular_1995} and Manoharmayum \cite{manoharmayum_modularity_2001}. The basis of such methods is the possibility to find lots of solvable points on modular curves, and as such is known to apply to very few modular curves. Details of the process is content of section \ref{Residual modularity}.

Having done the first two steps, we are left with understanding elliptic curves that we fail to establish modularity. These curves are naturally interpreted as points on some special modular curves, to which we can either establish qualitative statements, or if lucky, complete determination: after some preparatory material in section \ref{q-expansion}, section \ref{Gonality} gives studies these sets for general totally real fields, section \ref{Modularity over real quadratic fields} reduces the modularity problem (under simplifying assumptions) to studying quadratic points on some modular curves, and section \ref{Quadratic Points} is devoted to studying those quadratic points. This is the content of the remaining sections. We have made extensive use of the computer algebra system Magma \cite{bosma_magma_1997} to perform our computations. 

\subsection*{Acknowledgement}
I would like to thank my advisor Richard Taylor for suggesting this problem, for his constant encouragements and many enlightening discussions. I would also like to thank Frank Calegari for suggesting the possibility of switching primes at 7, as well as pointing me to the reference \cite{manoharmayum_modularity_2001}. I have also benefitted greatly from discussions with George Boxer, Noam Elkies, Bjorn Poonen, Maarten Derickx.   

After most of this work was done, Samir Siksek has informed me about his joint work with Freitas \cite{freitas_modularity_2013}, where they obtain modularity results about elliptic curves with full 2-torsion by similar methods.
\section{Modularity lifting}\label{Mod lift}
We record the following modularity lifting statement which is optimized for our purposes:
\begin{thm} \label{Modularity lifting}
  Let $p>2$ be prime, $F$ a totally real field, $\rho: G_F \to GL_2(\mathbb{Q}_p)$ a Galois representation. Assume 
  \begin{itemize}
    \item 
    $\rho$ is unramified for almost all places $v$ of $F$. 
    \item 
    For each place $v | p$, $\rho_v= \rho | _{G_{F_v}}$ is potentially semi-stable of with $\tau$-Hodge-Tate weight 0, 1 for each embedding $\tau: F_v \hookrightarrow \overline{\mathbb{Q}}_p$. 
    \item 
  $\det\rho \cong \chi_{cyc}$ is the cyclotomic character.
  \item 
  $\overline{\rho}|_{G_{F(\zeta_p)}}$ is absolutely irreducible.
 \item 
 $\overline{\rho}$ is modular of weight 2, that is there exists an automorphic representation of $\GL_2(\mathbb{A}_F)$ with trivial infinitesimal character with associated Galois representation $\rho_\pi$ such that $\overline{\rho}_\pi\cong \overline{\rho}$.
  \end{itemize}
Then $\rho$ is modular
\end{thm}
\begin{proof}
  This is the combination of various theorems in the literature. The case when $\rho$ is potentially crystalline is the main the result of \cite{kisin_moduli_2009}. To deal with the case $\rho$ is potentially semi-stable but not crystalline, we refer the reader to the proof of Theorem 3.2 in \cite{breuil_formes_2012}. 
 We note that in the reference cited there is an extra hypothesis when $p=5$. The presence of this hypothesis is to assure the existence of Taylor-Wiles systems \cite{kisin_moduli_2009} 3.2.3. However, under our assumption that the determinant of $\overline{\rho}$ is the cyclotomic character, we can still choose Taylor-Wiles systems without this extra hypothesis. With the notation in \cite{kisin_moduli_2009} 3.2.3 and following the proof of Theorem 2.49 in \cite{darmon_fermats_1994}, what we need to show is that we can find for each $n$ and a non-trivial cocycle $\psi\in H^1(G_{F,S},\textrm{ad}^0 \overline{\rho}(1))$, we can find a place $v\notin
  S$ of $F$ such that 
\begin{itemize}
  \item
$|k(v)|=1$ mod $p^n$ and $\overline{\rho}(Frob_v)$ has distinct eigenvalues.  
\item   
The image of $\psi$ under the restriction map 
\begin{align*}
  H^1(G_F,\textrm{ad}^0 \overline{\rho}(1))\to H^1(G_{F_v},\textrm{ad}^0 \overline{\rho}(1))
\end{align*}
is non-trivial.
  \end{itemize}
 Let $F_m$ be the extension of $F(\zeta_{p^m})$ cut out by $\textrm{ad}^0\overline{\rho}$, then the argument in \cite{darmon_fermats_1994} works once we can show that the restriction of $\psi$ to $H^1(G_{F_0},\textrm{ad}^0\overline{\rho}(1))$ is non-trivial. To do this we want to show that $H^1(\textrm{Gal}(F_0/F),\textrm{ad}^0\overline{\rho}(1)^{G_{F_0}})=0$. Because $G_{F_0}$ acts trivially on $\textrm{ad}^0\overline{\rho}$, the coefficient module vanishes unless $\zeta_p\in F_0$. Because $H^1(\textrm{PSL}_2(\mathbb{F}_5),\textrm{Sym}^2\mathbb{F}_5^2)=\mathbb{F}_5$ does not vanish, the extra hypothesis when $p=5$ was needed to exclude the possibility that $\textrm{Gal}(F_0/F)=\textrm{PSL}_2(\mathbb{F}_5)$ and $\zeta_5\in F_0$. But under our cyclotomic determinant assumption we can get rid of it as follows: Suppose the above situation happens, then because $\det:\overline{\rho}(G_F) \subset \textrm{PGL}_2(\mathbb{F}_5)\to \mathbb{F}_5^\times/(\mathbb{F}_5^\times)^2$ is induced by the cyclotomic character, we see that $\sqrt{5}\in F$. But $\zeta_5\in F_0\setminus F$ implies $\textrm{Gal}(F_0/F)=\textrm{PSL}_2(F_5)=A_5$ has a quotient of order $2$, a contradiction.
\end{proof}
\begin{remark}
  The lifting theorem in \cite{kisin_moduli_2009} requires one to have a modular lift which lies in the same connected component (''ordinary'' or ''non-ordinary'') as $\rho$ at all places above $v$. Corollary (3.1.6) in \cite{kisin_moduli_2009} guarantees such a lift once we have an ordinary automorphic lift of $\overline{\rho}$. In general, one can appeal to \cite{barnet-lamb_congruences_2010}, \cite{barnet-lamb_congruences_2012} for the existence of ordinary lifts, however in situations when we apply theorem \ref{Modularity lifting}, we could always guarantee an ordinary automorphic lift (either by Hida theory or by choosing our auxilliary elliptic curves in the argument below judiciously).  
\end{remark}   
\section{Residual modularity and Prime switching} 
\label{Residual modularity}
If $\rho_{E,p}$ is a Galois representation coming from the Tate module of an elliptic curve over $E$, then the first three items of \ref{Modularity lifting} are satisfied, and it remains to consider the last two items. In particular, we need to have access to enough mod $p$ modular Galois representations (of weight 2).
The basic starting point is
\begin{thm}(Langlands-Tunnell \cite{langlands_base_1980}, \cite{tunnell_artins_1981})
If $F$ is a totally real field and $\rho: G_F \to \GL_2(\mathbb{C})$ is an odd Artin representation with solvable projective image then $\rho$ is modular.
  \end{thm}
  Using this, the argument in \cite{wiles_modular_1995} shows that if $E$ is an elliptic curve $F$ then $\overline{\rho}_{E,3}$ is congruent to the Galois representation $\rho_\pi$ associated to a Hilbert modular form of weight $1$. Such a representation is ordinary at all places $v| 3$, and hence \cite{wiles_ordinary_1988} shows that $\rho_\pi$ is obtained by the specialization of a Hida family to parallel weight $1$. Specializing the family at parallel weight $2$ then shows that $\overline{\rho}_{E,3}$ is in fact modular of weight $2$. Thus $\overline{\rho}_{E,3}$ is always modular of weight $2$. Starting from this, we can propagate residual modularity:
  \begin{prop}
    
    Let $E$ be an elliptic curve over a totally real field $F$. Then
    \begin{enumerate}
      \item 
      There exists an elliptic curve $E'$ over $F$ such that 
      \begin{itemize}
        \item 
        $\overline{\rho}_{E,5}\cong \overline{\rho}_{E',5}$
        \item
        $\mathrm{Im} \,       \overline{\rho}_{E',3} \supset \mathrm{SL}_2(\mathbb{F}_3)$        
          \end{itemize}
          In particular, $\overline{\rho}_{E,5}\cong \overline{\rho}_{E',5}$ is modular of weight 2.
      \item 
       There exists an elliptic curve $E'$ over a solvable extension $F'$ of $F$ such that 
      \begin{itemize}
        \item 
        $\overline{\rho}_{E,7}| _{G_{F'}}\cong \overline{\rho}_{E',7}$
        \item
        $\mathrm{Im}\,\overline{\rho}_{E',7}=\mathrm{Im}\,\overline{\rho}_{E,7}$
        \item
        $\mathrm{Im}\,\overline{\rho}_{E',3} \supset \mathrm{SL}_2(\mathbb{F}_3)$        
          \end{itemize}
           \end{enumerate}
  \end{prop}
  
  \begin{proof} 
   Given an elliptic curve $E$ over $F$, we have a finite (\'{e}tale) group scheme $E[p]$ over $F$, and thus we have a twisted modular curve $X_{E}(p)$ defined over $F$ which classifies isomorphism classes of (generalized) elliptic curves $E'$ together with a symplectic isomorphism of group schemes
  	   \[E[p]\cong E'[p] \] 
     Indeed one has such a twisted modular curve over $F$ for any Galois representation $\overline{\rho}:G_F \to \GL_2(\mathbb{F}_p)$ with cyclotomic determinant, by replacing $E[p]$ with the group scheme $\mathcal{G}$ with descent data given by $\rho$. 
   \begin{enumerate}
    \item (see \cite{wiles_modular_1995})
    The curve $X_E(5)$ has genus $0$ and has a rational point (corresponding to $E$), hence is isomorphic to $\mathbb{P}^1$ over $F$. The variety parameterizing of bases $(P,Q)$ of the 3-torsion subscheme of the universal elliptic curve over $X_E(5)$ has $2$ geometric connected components, which are covers of $X_E(5)$ of degree $|\mathrm{SL}_2(\mathbb{F}_3)|$. Hence Hilbert's irreducibility theorem shows that one can find an $F$-rational point on $X_E(5)$, corresponding to an elliptic curve $E'$ such that the field cut out by $E'[3]$ has degree $\geq 24$ over $F$, giving the desired curve because $\mathrm{SL}_2(\mathbb{F}_3)$ is the unique proper subgroup of $\mathrm{GL}_2(\mathbb{F}_3)$ of size $\geq 24$.
Theorem \ref{Modularity lifting} then shows that $E'$ is modular, hence the last assertion.
    \item 
    This is a more elaborate version of the above argument. We will follow the general approach in \cite{manoharmayum_modularity_2001}, giving some further details. The twisted modular curve for the mod 7 Galois representation $\chi_{cyc}\oplus 1$ is isomorphic (over $F$) to the Klein quartic
    \[ X^3Y+Y^3Z+Z^3X=0\]
    in $\mathbb{P}^2$.
  We denote by $C$ the twisted modular curve $X_E(7)$. $C$ is a smooth non-hyperelliptic curve of genus $3$, hence the canonical embedding realizes $C$ as a plane quartic $C\hookrightarrow \mathbb{P}^2$. Let $L$ denote the space of lines inside this $\mathbb{P}^2$, then $L$ is the dual projective space and is isomorphic to $\mathbb{P}^2$. Let $Z \hookrightarrow C^4$ be the subvariety consisting of ordered quadruple of points on $C$ that are collinear. Note that $Z=L\times_{\mathrm{Sym}^4 C} C^4 $, where $\mathrm{Sym}^4 C=C^4/S_4$ is the fourth symmetric power of $C$. By lemma \ref{Collinear 1} $Z$ is geometrically irreducible and the cover $Z\to L$ is generically Galois with Galois group $S_4$.
  
  Let $T$ denote the variety parameterizing bases $(P,Q)$ for the 3-torsion subscheme of the universal elliptic curve over $C$, it is a $\GL_2(\mathbb{F}_3)$-torsor over the complement of the cusps on $C$. Define $Y= Z \times_{C^4} T^4$. By lemma \ref{Collinear 2}, each geometric irreducible component of $Y$ has degree $\geq 24^5/3$ over $L$. Because $L$ satisfies weak approximation, the main result of \cite{ekedahl_effective_1990} shows that the subset of $F$-rational lines $l$ such that the fibers in $Y$ over $l$ each have degree $\geq 24^5/3$ and the fiber in $Z$ over $l$ is connected also satisfies weak approximation. If a line $l$ has the above properties, $l$ intersects $C$ at four points whose residue field is an $S_4$ Galois extension of $F'$ of $F$. The 4 intersection points give rise to 4 elliptic curves $E_i$ over $F'$, which are conjugates of each other, hence the degree $d$ of the extension $F'(E_i[3])$ obtained by adjoining the 3-torsion points of $E_i$ is independent of $i$. On the other hand, the extension over $F'$ generated by adjoining all the 3-torsion points of all the $E_i$ has degree $\geq 24^4/3$ over $F'$, thus we have $d^4\geq 24^4/3$. But any subgroup of $\GL_2(\mathbb{F}_3)$ of such size $d$ must contain $\mathrm{SL}_2(\mathbb{F}_3)$. Thus to finish the proof, we only need to arrange that the intersection points are defined over a totally real extension of $F$, and that the mod $7$ Galois representation of the elliptic curves corresponding to the intersection points have as large image as $\mathrm{Im}\, \overline{\rho}_{E,7}$. This will be done by using the weak approximation property to find a line $l$ as above which lies in open subsets $U_v$ for $v$ running over a finite set of places of $F$ chosen as follows:
\begin{itemize}
  \item
For each $v|\infty$, because $\overline{\rho}_{E,7}$ has cyclotomic determinant, \\$X_E(7)\times_F F_v$ is actually isomorphic to the Klein quartic, and thus we find an explicit line $l_v$ which intersects $C$ at 4 non-cuspidal real points, for example the line $3X+Y=0$. Every line in a small open neighborhood (for the strong topology) of $l_v$ will then have the same property. Shrinking $U_v$, we can assume it contains no line passing through a cusp.
\item   
By the Chebotarev density theorem, for each element $g\in \mathrm{Im}\,\overline{\rho}_{E,7}$, there are infinitely many places $v$ such that $\overline{\rho}|_{G_{F_v}}$ is unramified and the image of $Frob_v$ is $g$. Pick such a $v$ for each $g$, such that $C$ has good reduction and reduces to a smooth plane quartic $\overline{C}$. If a Galois extension $F'$ of $F$ is such that all places $v$ in this list splits completely, then $F'$ is linearly disjoint from $F^{\mathrm{ker} \overline{\rho}}$. By lemma \ref{4 rational intersection points}, if we pick $v$ with $\mathrm{N}v>300$, the reduction $\overline{C}$ will contain 4 collinear rational points. By Hensel's lemma, any line that reduces to the line going through these 4 points will intersect $C\times_F F_v$ at 4 $F_v$- rational points. This gives an open subset $U_v$ of $L(F_v)$ all whose members have intersect $C$ at $F_v$-rational points and does not contain a line passing through a cusp.
\end{itemize}  
It is now clear that if an $F$-rational line $l$ is in $U_v$ for the above choice, the 4 intersection points of $l$ with $C$ will have coordinates in a totally real extension $F'$ of $F$, and the image of the mod $7$ representation corresponding to each intersection point is the same as that of $E$.   
    \end{enumerate}
  \end{proof}
  \begin{lem}
      \label{Collinear 1}
      Let $C$ be the Klein quartic over $\mathbb{C}$. Let $L=\mathbb{P}H^0(C,\Omega^1)\hookrightarrow \mathrm{Sym}^4 C$ and $Y=L\times_{\mathrm{Sym}^4 C}C^4$ be the space of ordered quadruple of collinear points on $C$. Then $Y$ irreducible and $Y\to L$ is generically Galois with Galois group $S_4$.  
    \end{lem}  
    \begin{proof}
    Let $U\subset L$ denote the locus over which $Y\to L$ is \'{e}tale.
     We pick two flexes $P_0$, $P_1$ on $C$ such that the line $l_0=l(P_0P_1)$ intersects $C$ at four distinct points $P_0$, $P_1$, $Q$, $R$. Pick the base point $y_0=(P_0,P_1,Q,R)\in 
     Y$ which lifts $l_0$. We have the diagram
     \begin{align*}
      \xymatrix{ Y_{P_0}\ar[r]\ar[d]^{pr_2} &Y \ar[d]\\
     C \ar[d] & L \\
     L_{P_0} \ar[ur] &
     }
    \end{align*}
     where $L_{P_0}$ is the $\mathbb{P}^1$ of lines in $L$ passing through $P^1$, $Y_{P_0}$ is the subvariety of $Y$ consisting of quadruples whose first coordinate is $P_0$, and $pr_2$ denotes the projection to the second coordinate. Because there is a line through $P_1$ which is tangent to $C$ and has contact order $3$ at a point different from $P_0$, and the tangent line to $C$ from $P_0$ meets $C$ at exactly another point other than $P_0$, we see that the cover $C\to L_{P_0}$ in the diagram is a triple cover which has both simple and triple ramification points. Consequently, the monodromy group of the cover is all of $S_3$. Thus there are loops in $\pi_1(U,l_0)$ whose lift to $Y$ provides a path that connects $y_0$ to any other point in the fiber over $l_0$ whose first coordinate is $P_0$. But applying the same argument for $P_1$, we can also construct paths joining $y_0$ and points in the fiber over $l_0$  whose second coordinate is $P_1$. It follows that the monodromy group of the cover $Y_U\to U$ must be a subgroup of $S_4$ which contains the stabilizer of two different points, and hence must be all of $S_4$, showing that $Y$ is irreducible and the generic Galois group is $S_4$.
    \end{proof} 
    
    \begin{lem}
      \label{Collinear 2}
      Let $C=X(7)$ be the modular curve with full level 7 structure over $\mathbb{C}$. Fix a non-trivial $\zeta\in\mu_3(\mathbb{C})$, and let $T_\zeta$ denote the space parameterizing bases $(P,Q)$ of the 3-torsion subgroup of the universal elliptic curve over $C$, such that the Weil pairing $e_3(P,Q)=\zeta$. Let $Z\hookrightarrow C^4$ denote space of ordered quadruples which are collinear (under the canonical embedding of $C$). Put $Y=Z\times_{C^4} T_\zeta^4$. Then each irreducible component of $Y$ has degree $\geq 24^4/3$ over $Z$.  
      \end{lem}
      \begin{proof}
      Put $G=\mathrm{SL}_2(\mathbb{F}_3)$. $\widetilde{G}=G^4\rtimes S_4$, where $S_4$ acts by permuting the coordinates on $G^4$. Let $L=\mathbb{P}H^0(C,\Omega^1)\hookrightarrow \mathrm{Sym}^4 C$ as in the previous lemma.
      The cover $Y \to L$ is generically \'{e}tale with Galois group $\widetilde{G}$. Let $Y_0$ be an irreducible component of $Y$, then $Y_0\to L$ is generically \'{e}tale with Galois group a subgroup $\widetilde{H}\subseteq \widetilde{H}$, which surjects onto $S_4$. Let $H=\widetilde{H}\cap G^4$.Now for each pair $(i,j)$ with $1\leq i,j\leq 4$ we have a comutative diagram
      \begin{align*}
     \xymatrix
{	Y_0 \ar[d] \ar[r] & T_\zeta^4 \ar[d]\ar[r]^{pr_{ij}} & T_\zeta^2 \ar[d]\\
	Z  \ar[r] & C^4 \ar[r]^{pr_{ij}} & C^2 }      
      \end{align*}
       The composition map $Z\to C^2$ is generically a 2 to 1 cover. Because $G/[G,G]\cong\mathbb{Z}/3\mathbb{Z}$, $G\times G$ has no non-trivial homomorphism to $\mathbb{Z}/2\mathbb{Z}$, hence the function fields $\mathbb{C}(Z)$ and $\mathbb{C}(T_\zeta^2)$ are linearly disjoint over $\mathbb{C}(C^2)$. It follows that   
     via the projection to the $(i,j)$ factors, there is a surjection $H\twoheadrightarrow G^2$. Let us now consider the image $\overline{H}$ of $H$ under the projection to any 3 coordinates $G^3$. Then for each $a,b \in G$, there is some $(a,1,\phi_1(a))\in \overline{H}$ and $(b,\phi_2(b),1)\in \overline{H}$. Thus the commutator $(aba^{-1}b^{-1},1,1)\in \overline{H}$. It follows that $\overline{H} \supseteq [G,G]^3$. Now for any $a \in [G,G]$, $b\in G$ we have some
$(a,1,1,\psi_1(a))\in H$ by what we just proved, and some $(b,\psi_2(b),\psi_3(b),1) \in G$. Taking commutators and noting that $[G,[G,G]]=[G,G]$ we see that $H\supset [G,G]^4$ strictly. But passing to $(G/[G,G])^4 \rtimes S_4$ which is an $\mathbb{F}_3$-vector space, the image of $H$ correspond to a subspace whose projection to any two coordinates are surjective. Furthermore this image must also be stable under permutation of the coordinates, so it must contain the hyperplane $a+b+c+d=0 \in \mathbb{F}_3$. In particular, the index of $H$ in $G^4$ is at most $3$.
      \end{proof}
    \begin{lem}
    \label{4 rational intersection points}(see \cite{oyono_rationality_2010}) Let $C$ be a smooth plane quartic defined over $\mathbb{F}_q$. If $q>300$, then there exists a line $l$ that intersects $C$ at 4 distinct rational points.

    \end{lem}
\section{Finding \textit{q}-expansions}
\label{q-expansion}
The aim of this section is to describe a method to compute the $q$-expansion of modular forms that correspond to holomorphic differentials on certain modular curves with non-standard level structures. Except for lemma \ref{Jacobian decomposition}, the rest of the section is only used to obtain $q$-expansions needed for computations on the curves $X(3b,5s^+)$ and $X(3ns^+,7s^+)$ in section \ref{Quadratic Points}. The reader who takes such computations on trusts can therefore skip the rest of this section. As it did not seem easy to extract these sort of computations in the literature, we decided to record it here, and hope that it might be useful for other works. 

Let $\Gamma$ be a congruence subgroup of $\mathrm{SL}_2(\mathbb{Z})$ of square-free level $N$ such that the image of $\Gamma$ mod p is either the Borel, the normalizer of the split or non-split Cartan subgroup of $\mathrm{SL}_2(\mathbb{F}_p)$ for each prime $p | N$.
The modular curve $X(\Gamma)=\mathbb{H}^*/\Gamma$ is thus defined over $\mathbb{Q}$. This is because for each prime $p$ dividing the level, the corresponding subgroup in $\mathrm{SL}_2(\mathbb{F}_p)$ admits an extension to a subgroup of $\GL_2(\mathbb{F}_p)$ with determinant surjecting onto $\mathbb{F}_p^\times$, and hence the corresponding open modular curve is identified with the Shimura variety $\Sh_K=\GL_2(\mathbb{Q})\setminus \GL_2(\mathbb{A})/\mathbb{R}^\times\mathrm{O}(\mathbb{R})K$, where $K$ is the unique open compact subgroup of $\GL_2(\widehat{\mathbb{Z}})$ lifting each subgroup of $\GL_2(\mathbb{F}_p)$ as above. Note because $\det(K)$ is surjective, the Shimura variety is geometrically irreducible, and its $\mathbb{C}$-points are naturally given by $\mathbb{H}/\Gamma$.

Given a list of distinct primes $p_i$ and label *="$b$", "$s^+$" or "$ns^+$", we will write $X(p_i*)$ to denote the modular curve of the above kind such that mod $p_i$ the congruence subgroup is Borel, normalizer of split or non-split Cartan respectively. 

Fixing a prime $p$, denoting the Borel, Normalizer of split/non-split Cartan subgroups of $G=\GL_2(\mathbb{F}_p)$ by $B$, $S$ and $N$ (and the corresponding open compact subgroup of $\mathrm{GL}_2(\mathbb{Z}_p)$ by $B_p$, $S_p$, and $N_p$), we have a relation \cite{edixhoven_result_1996}
\[  \pi_N+\pi_B= v\pi_S v^{-1}+\pi_G   \]
inside the group algebra $\mathbb{Q}[G]$, where $\pi_H$ denotes the projector onto the $H$-invariant parts, and $v$ is an invertible element in $\mathbb{Q}[G]$. 
By applying this relation onto $\End(\Sh_{K(N)})\otimes \mathbb{Q}$, we can thus express (up to isogeny) the Jacobian of each modular curve of the kind we are considering in terms of Jacobian the modular curves of the same kind, but where only Borel or normalizer of split Cartan level structures appear. Furthermore, because the open compact $K(p^{s})\subset \GL_2(\widehat{\mathbb{Z}}_p)$ satisfies
\begin{align*}
    K(p^{s})&=   \left(\begin{smallmatrix}
        p & 0 \\ 
        0 & 1
      \end{smallmatrix}\right)  K_0(p^2)  \left(\begin{smallmatrix}
        p & 0 \\ 
        0 & 1
      \end{smallmatrix}\right)^{-1}   \\
K(p^{s+})&=   \left(\begin{smallmatrix}
        p & 0 \\ 
        0 & 1
      \end{smallmatrix}\right) \langle K_0(p^2),  \left(\begin{smallmatrix}
        0 & \frac{-1}{p} \\ 
        p & 0
      \end{smallmatrix}\right)  \rangle \left(\begin{smallmatrix}
        p & 0 \\ 
        0 & 1
      \end{smallmatrix}\right)^{-1}  , 
       \end{align*} 
the curves with normalizer split Cartan level at $p$ are isomorphic to one with level $K_0(p^2)^+$.

The modular curve $\mathrm{Sh}_{K_0(p^m)K^p}$ has a moduli interpretation in terms of elliptic curves with a cyclic subgroup $C_{p^m}$ of order $p^m$ and level structure $K^p$ away from $p$, and the Atkin-Lehner involution $w_{p^m}$ given by
\begin{align*}
  (E,C_{p^m},K^p) \to (E/C_p,E[p^m]/C_{p^m},K^p)
\end{align*}
corresponds to conjugation action by $\left(\begin{smallmatrix} 0 &-1 \\ p^m &0 \end{smallmatrix}\right)$.
To summarize we have
\begin{lem} \label{Jacobian decomposition} If $p$ is any prime and $K^p$ is a level structure away from $p$, up to isogeny over $\mathbb{Q}$
\begin{align*}
  \mathrm{Jac}(\mathrm{Sh}_{N_pK^p})\times \mathrm{Jac}(\mathrm{Sh}_{B_pK^p}) &\sim \mathrm{Jac}(\mathrm{Sh}_{S_pK^p}) \times \mathrm{Jac}(\mathrm{Sh}_{\GL_2(\mathbb{Z}_p)K^p}) \\
  \mathrm{Jac}(\mathrm{Sh}_{S_pK^p}) &\sim \mathrm{Jac}(\mathrm{Sh}_{K_0(p^2)K^p})^{w_{p^2}}
\end{align*}
  \end{lem}

Thus all isogeny factors of the modular curves we consider are expressible as factors of the standard modular Jacobian $J_0(N)$ and hence are modular abelian varieties attached to suitable newforms $f$ with trivial Nebentype.

  Suppose we want to find the $q$-expansion of the weight 2 modular forms on the modular curvesof the above type. By the above, this reduces to the question of finding vectors fixed under the relevant open compact subgroups given the new vector $f$ in an automorphic representation $\pi_f$ of $\GL_2(\mathbb{A})$ that contributes to the Jacobian of our modular curve, in a way that is easy to extract the $q$-expansion in terms of the $q$-expansion of $f$ (which one can then look up in \cite{mfd}). This reduces to a problem of local representation theory.
  
  Recall that $\pi_f= \bigotimes_p \pi_{f,p}$ has a decomposition into smooth irreducible representations $\pi_{f,p}$ of $\GL_2(\mathbb{Q}_p)$. Since it contributes to the Jacobian of some $X_0(M)$, it must have trivial central character. We also note that the newforms $f$ contributing to our modular curves have conductor at most 2, so at a bad prime $\pi_{f,p}$ is either a supercuspidal representation of conductor 2, a twist of a conductor 1 principal series by a conductor 1 character, or a twist of the Steinberg representation.
  \begin{prop}
  In the above situation, the $\GL_2(\mathbb{F}_p)$-representation of $\pi_{f,p}^{K(1)}$ is irreducible, and is supercuspidal, principal series or Steinberg according to $\pi_{f,p}$.
  \end{prop}
  \begin{proof}
    In the supercuspidal case, this follows from section 3 of \cite{loeffler_computation_2012}. In the principal series case, $\pi_{f,p}\cong \mathrm{Ind}\, \chi\otimes\chi^{-1}$ with  $\mathrm{cond}\chi$=1, and the explicit description shows that $\pi_{f,p}^{K(1)}=\mathrm{Ind}\,\overline{\chi}\otimes\overline{\chi}^{-1}$, where $\overline{\chi}$ is the reduction of $\chi$. The Steinberg case also follows from this.
    \end{proof}
    We will now describe how to find the fixed vectors under the relevant open compact subroups according to each case.
\subsection{The principal series case}
In this section, we abbreviate $\pi=\mathrm{Ind}\,\chi\otimes\chi^{-1}$ for the principal series representation of $G=\GL_2(\mathbb{F}_p)$ given by a character $\chi:\mathbb{F}_p^\times\to \mathbb{C}^\times$. Let $B$, $N$, $T$, $T'$  denote the subgroups of upper triangular, unipotent upper triangular, diagonal and a non-split torus in $G$. A choice of $T'$ is fixed by choosing a non-square $a\in \mathbb{F}_p$, and taking the image of $\mathbb{F}_{p^2}^\times$ inside $G$ via the multiplication action on $\mathbb{F}_{p^2}\cong \mathbb{F}_p^2$ where the last isomorphism comes from the choice of basis $1$, $\alpha=\sqrt{a}$. With this choice, $T'$ consists of matrices of the form 
\[\left(\begin{smallmatrix} u &av\\ v &u 
\end{smallmatrix}\right)\]

Recall that the vector space underlying $\pi$ is the space of functions $f:G\to \mathbb{C}$ such that
\[ f(bg)=\chi\otimes\chi^{-1}(b)f(g)\]
and G acts on $\pi$ by right translation. The coset decomposition
\[G= B \amalg_{x\in\mathbb{F}_p} B  \left(\begin{smallmatrix}
        0 & 1 \\ 
        1 & x
      \end{smallmatrix}\right)   \]
identifies the vector space underlying $\pi$ with the space of functions on $\mathbb{P}^1(\mathbb{F}_p)$. The coset $B$ corresponds to $\infty$.

Computing characters, we see that as a representation of $B$, $\pi$ decomposes as
\[ \chi\otimes \chi^{-1} \oplus \chi^{-1}\otimes \chi \oplus V  \]
where $V$ is the unique $p-1$-dimensional irreducible representation of $B$ with trivial central character. As a representation of $N$, $V$ decomposes into lines
\[V=\oplus V_\psi\]
as $\psi$ runs through the non-trivial characters of $N$.

The function $\delta_\infty$ gives a generator for the line $\chi\otimes\chi^{-1}$, and $\delta_0=\chi(-1)w\delta_\infty$, where $w=\left(\begin{smallmatrix} 0 & -1\\ 1 &0
\end{smallmatrix}\right)$
As functions on $\mathbb{F}_p\subset \mathbb{P}^1(\mathbb{F}_p)$ we have
\[\delta_0(x)=\frac{1}{p}\sum_\psi \psi(x),\]
thus one recovers generators $v_\psi$ (corresponding to the function $\frac{1}{p}\psi$) of $V_\psi$ as the $\psi$-isotypic components of $\delta_0$.

Now suppose $f$ is a function invariant under $T'$. Then
\begin{align*}
   f(\left(\begin{smallmatrix} u &av\\ v &u \end{smallmatrix}\right))= f(\left(\begin{smallmatrix} 1 &0\\ 0 &0\end{smallmatrix}\right)) \end{align*}
so for $v\neq 0$
\[
f(\left(\begin{smallmatrix} av-\frac{u^2}{v} &u\\ 0 &v  \end{smallmatrix}\right) \left(\begin{smallmatrix} 0 &1\\ 1 &\frac{u}{v} \end{smallmatrix}\right))=f(\left(\begin{smallmatrix} 1 &0\\ 0 &1 \end{smallmatrix}\right))
\]
and hence as a function
\[ f(x)=\chi(a-x^2)^{-1}f(\infty)\]
One can check that conversely, this function gives a $T'$-fixed vector.
This show that $\pi^{T'}$ is one-dimensional, and we can write down a non-zero element in it by
\[f= \delta_\infty+\displaystyle\sum_{\psi:\mathbb{F}_p\to \mathbb{C}^\times} A_\psi v_\psi,\]
with 
\[A_\psi=\displaystyle\sum_{x\in\mathbb{F}_p}  \chi(a-x^2)^{-1}\overline{\psi}(x)\]

Combining this and a similar computation for $T$ we get
\begin{prop}
Let   $\pi=\mathrm{Ind}\, \chi\otimes\chi^{-1}$, $\chi\neq 1$.
\begin{itemize}
  \item 
  $\dim \pi^T=1$, and $\pi$ admits $N(T)$-fixed vectors if and only if $\chi(-1)=1$.
  \item 
  $\dim \pi^{T'}=1$, and $\pi$ admits $N(T')$-fixed vectors if and only if $\chi(-1)=1$.
  \end{itemize}
 
\end{prop}
\subsection{The cuspidal case}\label{Cuspidal formula}
This case is dealt with in \cite{Bar}.
Keeping the notation in the previous section, except that $\pi$ is now a cuspidal representation of $G$ with trivial central character, corresponding to a character $\theta:\mathbb{F}_{p^2}^\times\to\mathbb{C}^\times$.

As a representation of $B$, $\pi$ is the unique irreducible $p-1$-dimensional representation with trivial central character. As a representation of $N$, it decomposes into the sum of all the non-trivial additive characters $\psi$ of $N$.
Fixing such a non-trivial additive character, $\pi$ has a model on the space of functions on $\mathbb{F}_p^\times$ \cite{piatetski-shapiro_complex_1983} (the "Kirillov model") with the action of $G$ given by
\[\left(\begin{smallmatrix} u &v \\0 &w \end{smallmatrix}\right) f(x)=\psi((v/w)x)f((u/w)x)
 \]
 and
 \[ g.f(x)= \displaystyle\sum_{y\in\mathbb{F}_p\times} k(x,y;g)f(y) \]
where $g=\left(\begin{smallmatrix} u &v \\z &w\end{smallmatrix}\right)$ and
\[k(x,y;g)= -\frac{\psi(ux+wy/z)}{p}\displaystyle\sum_{\mathrm{N}(t)=(x/y)\det(g)} \psi(-(y/z)\mathrm{Tr}(t))\theta(t) \]

The vector space underlying $\pi$ has a basis consisting of $\delta$ functions, and the $k(x,y;g)$ gives the coefficients of the matrix of g acting with respect to this basis. Note that the function $\displaystyle\sum_x \delta_x$ is clearly the unique function up to scaling that is fixed under $T$.

Because $\pi$ has trivial central character, a $T'$-fixed vector can be obtained by 
\begin{align*}
&\quad \delta_1+\displaystyle\sum_{u}\left(\begin{smallmatrix} u &a \\1 &u\end{smallmatrix}\right)\delta_1 \\
&=\delta_1+\displaystyle\sum_{u}\sum_{x\in\mathbb{F}_p^\times}k(x,1;\left(\begin{smallmatrix} u &a \\1 &u\end{smallmatrix}\right))\delta_x \\
&=\delta_1+A_{\theta,x}\delta_x
  \end{align*}
where
\begin{align*}
  A_{\theta,x}&=\displaystyle\sum_u k(x,1;\left(\begin{smallmatrix} u &a \\1 &u\end{smallmatrix}\right))
\\ &=-\frac{1}{p}\displaystyle\sum_{u\in\mathbb{F}_p}\sum_{\mathrm{N}(t)=x(u^2-a)} \psi(u(x+1)-\mathrm{Tr}(t))\theta(t)
\end{align*}
In summary, starting with a $T$-fixed vector $f$ in $\pi$, the above gives a formula for the unique $T'$-fixed vector in $\pi$. Note that all operations we needed to perform have simple effect on $q$-expansions, as we only needed to decompose $f$ into isotypic parts for $N$. 
\subsection{$q$-expansions}
Suppose we start with the $q$-expansion of a newform $f\in S_2(\Gamma_0(N))$, and for each prime $p$ either $p||{N}$ or $p^2|| N$ and we choose a label "$s^+$" or "$ns^+$". We wish to find the $q$-expansion of the (unique up to scaling, if it exists) form that correspond to the fixed vector under the open compact that has the corresponding label at each $p| N$. It turns out to be more convenient to work with the form that is $p$-primitive for each prime $p$ where $\pi_{f,p}$ is ramified principal series of conductor 2 and the label of $p$ is "$ns^+$. Thus we will take as input a newform $f\in S_2(\Gamma(N),\chi^2)$ where $\chi$ is a Dirichlet character of square-free conductor $M$ such that $\mathrm{gcd}(M,N/M)=1$.

For any Dirichlet character $\chi$, we have a unique adelization
\[\omega_{\chi}: \mathbb{Q}^\times\setminus \mathbb{A}^\times\to \mathbb{C}^\times \]
by requiring $\omega_{\chi,p}(\varpi_p)=\chi(p)$ for all $(p,M)=1$ and uniformizer $\varpi_p$ at $p$. This implies the restriction of $\omega_\chi$ to $\widehat{\mathbb{Z}}^\times$ factors through $\chi^{-1}:(\mathbb{Z}/M\mathbb{Z})^\times \to \mathbb{C}^\times$.
We thus obtain a character
\[\omega: K_0(N):\to \mathbb{C}^\times \]
given by
\[\omega: \left(\begin{smallmatrix} a &b \\ c &d \end{smallmatrix}\right) \mapsto \omega_\chi((d/a)_N)\]
Note this is trivial on the central subgroup.

For such $f$ one associates the automorphic form
\[ \phi_f: \GL_2(\mathbb{Q})\setminus \GL_2(\mathbb{A})/Z(\mathbb{A}) \to \mathbb{C}\]
given by
\[\phi_f(\gamma g kz)=\omega(k)j(g_\infty,i)^{-2}\det(g)f(g_\infty i)
\]
for $\gamma \in \GL_2(\mathbb{Q}), g=(g_\infty,g_f)\in\GL_2(\mathbb{A}), k\in K_0(N), z \in Z(\mathbb{A})$, which gives a well-defined function because of the modularity condition on $f$. Here  $j( \left(\begin{smallmatrix}  a &b \\ c &d \end{smallmatrix}\right),z)=cz+d$. Note that 
\[ \phi_f(g_\infty k_\infty)=j(k_\infty,i)^{-2}\det(k_\infty)\]
or each $k_\infty\in \mathrm{O}(\mathbb{R})$. Let $\mathcal{A}_2$ denote the space of automorphic forms with this condition at $\infty$. It is a representation of $\GL_2(\mathbb{A}_f)$.

The $\GL_2(\mathbb{A}_f)$-span of $\phi_f$ inside $\mathcal{A}_2$ generates a smooth irreducible representation which is isomorphic to $\pi_f^{\infty}$, the finite part of the automorphic representation with trivial central character attached to $f$ (note this is a twist of the usual automorphic representation attached to $f$). The function $f$ correspond to a vector $\otimes_p v_p$ where $v_p$ is the new vector for each $(p,M)=1$, while for $p| M$, $\pi_{f,p}^{K(p)}=\mathrm{Ind}\,\chi_p\otimes \chi_p^{-1}$, and $v_p$ spans the unique subrepresentation of $\chi\otimes\chi^{-1}$ of $B(\mathbb{F}_p)$ inside it. We can thus find the invariant vector under the new open compact subgroup we want by performing the local operations in the previous section prime by prime. It remains to reinterpret this in terms of functions on $\mathbb{H}$ and their $q$-expansion.

There is a linear map
\[ cl: \mathcal{A}_2\to \mathrm{Maps}(\mathbb{H},\mathbb{C})  \]
given by
\[ \phi \mapsto z=g_\infty i \mapsto \phi((g_\infty,1))j(g_\infty,i)^2\det(g_\infty)^{-1}
\]
which is well-defined because of the automorphic condition at $\infty$.
This map is not injective, however if $\det: K\subset \GL_2(\widehat{\mathbb{Z}}) \to \widehat{\mathbb{Z}}^\times$ is surjective then the above maps restricted to the subspace of $K$-invariant (or more generally, eigenvectors for $K$) automorphic forms is injective, which holds in the case at hand. Given an element $g_f\in\GL_2(\mathbb{A}_f)$ and an automorphic form which is an eigenvector for some character $\omega$ of $K$, one has
\begin{align*}
&\quad cl(g_f.\phi)(g_\infty i)=\phi(g_\infty, g_f)j(g_\infty,i)^2\det(g_\infty)^{-1}\\
&=\phi((\gamma g_\infty,\gamma g_f))j(g_\infty,i)^2\det(g_\infty)^{-1}\\
&=\omega(\gamma g_f)cl(\phi)(\gamma g_\infty i) j(\gamma,g_\infty i)^{-2} \det(\gamma)\\
&=\omega(\gamma g_f)cl(\phi)|\gamma
\end{align*}
for any $\gamma \in \GL_2(\mathbb{Q})\cap Kg_f^{-1}$
\begin{itemize}
  \item 
  For $g_f=(..,1, \left(\begin{smallmatrix}  p &0 \\ 0 &1 \end{smallmatrix}\right),.. )$ where the matrix is at the component at $p$ and $K$ is such that $K_l$ contains the split Cartan open compact, we can choose $\gamma= \left(\begin{smallmatrix}  \frac{1}{p} &0 \\ 0 &1 \end{smallmatrix}\right)$ and the effect is given by
  \[ f \mapsto \omega^{(p)}\left(\begin{smallmatrix}  \frac{1}{p} &0 \\ 0 &1 \end{smallmatrix}\right)f(z/p)/p\]

  \item  For $g_f=(..,1, \left(\begin{smallmatrix}  1 &1 \\ 0 &1 \end{smallmatrix}\right),.. )$ we can choose  $\gamma= \left(\begin{smallmatrix}  1 &m \\ 0 &1\end{smallmatrix}\right)$ with $m=-1$ mod $p^r$, and $m=0$ mod $N/p^r$, where $N$ is the level of $K$.
\item 
 For $g_f=(..,1, \left(\begin{smallmatrix}  0 &-1 \\ 1 &0 \end{smallmatrix}\right),.. )$ we can choose
 $\gamma= \left(\begin{smallmatrix}  x & y \\ Mz &p^r w\end{smallmatrix}\right)\in \mathrm{SL}_2(\mathbb{Z})$ where $N=p^r M$ is the level of $K$, $p^r|| N$ such that $\gamma=1$ mod $M$ and $\gamma=\left(\begin{smallmatrix}  0 &-1 \\ 1 &0 \end{smallmatrix}\right)$ mod $p^r$. This means that $\gamma= w_{p^r} \left(\begin{smallmatrix}  1/p^r &0 \\ 0 &1 \end{smallmatrix}\right)$ where $w_{p^r}$ is a matrix defining the classical Atkin-Lehner operator (in the $\Gamma_0(N)$-level case).
 
\end{itemize}
\begin{example}
  Let $f\in S_2^{new}(\Gamma_0(5.7^2))$ be a newform with Atkin-Lehner eigenvalues 1 for $w_5$ and $w_{7^2}$. This form has Hecke field $\mathbb{Q}(\sqrt{2})$, and is unique up to Galois conjugacy. Using SAGE, we find that the local representation $\pi_{f,7}$ is cuspidal, whose corresponding parameter is the pair of characters of the unramified extension $\mathbb{Q}_7[s]/ (s^2 + 6s + 3 )$ with conductor 1 given by $s\mapsto d, 7 \mapsto 1$, where $d$ is a root of $x^2+\sqrt{2}x+1=0$. According to \cite{loeffler_computation_2012} Theorem 3.10 and the table in \cite{edixhoven_result_1996}, this means $\pi_{f,7}^K(7)$ is the cuspidal representation given by the reduction mod 7 of the above character on $\mathbb{Z}_7^\times$
  
  By the above, the form $f(z/7)$ correspond to a vector whose component at $7$ is fixed under the normalizer of the split Cartan open compact. We have
  \[ f(z/7)=\displaystyle\sum_{n>0} a_n t^n \]
  where $t=e^{2\pi i z/7}=q^{1/7}$. Thus if we write $f_i=\displaystyle\sum_{n=i \mathrm{ mod }7} a_n t^n$, then 
\[ f_i(z-1)=\zeta_7^{-i}f_i(z)  \]  
hence the $f_i$ are the $\psi^i$-isotypic part of $f(z/7)$ where $\psi:\mathbb{F}_7\to \mathbb{C}^\times$ is the character sending $1$ to $\zeta_7^{-1}$. Thus a vector fixed under $K(5b,7ns^+)$ in $\pi_f$ is given by 
\[ 
f_1-\frac{1}{7}\displaystyle \sum_{x\in\mathbb{F}_7^\times}\sum_{u\in \mathbb{F}_7}  \sum_{\mathrm{N}(t)=x(u^2+1)} \psi(u(x+1)-\mathrm{Tr}(t))\theta(t)
f_x  \]
\end{example}

The procedure we have is as follows:
\begin{itemize}
  \item Step 1: For primes $p$ where the label is "$s^+$" or "$ns^+$" but $f$ is supercuspidal at $p$, choose an ordering and successively replace $f$ with
  \[ \chi(p)^{-1}f(z/p)/p \]
  if $p^2|| N$
  and replace $f$ with
  \[ f+ \lambda_p f(z/p)/p \]
  if $p|| N$, and $\lambda_p$ is the eigenvalue of $w_{p}$ acting on $f$ (beware that this need not be $\pm 1$ if $f$ has non-trivial Nebentype).
  These operations have simple effects on the $q$-expansions. The outcome of this step is a vector $\otimes v_p$ where $v_p$ is fixed under the normalizer of the split Cartan at all primes where the modifications were made.
  \item 
  Step 2: For primes $p| M$, in order to use the result in the previous section we only need to find the $q-$expansion of $(..,1, \left(\begin{smallmatrix}  0 &-1 \\ 1 &0 \end{smallmatrix}\right),.. )\phi_f$. This commutes with all the operations done in step 1.
  First we determine this in the case $f$ was the newform we started with. Because $f$ has conductor $1$ at $p$ and non-trivial Nebentype, we have $a_p(f)\neq 0$, and a result of Atkin-Li \cite{atkin_twists_1978} Theorem 2.1 gives 
  \[ f| _{w_p}=\lambda \widetilde{f} \]
  where $\widetilde{f}$ is a newform whose Nebentype at $p$ is inverse of that of $f$ and whose $q$-expansion is explicitly read off $f$ and $\chi_p$. Furthermore the "pseudo-eigenvalue" $\lambda$ is given by
  \[\lambda= g(\chi_p)/a_p\]
  where $g(\chi_p)$ is the Gauss sum of the p-primary part $\chi_p$ of $\chi$.
  By the above, $(..,1, \left(\begin{smallmatrix}  0 &-1 \\ 1 &0 \end{smallmatrix}\right),.. )\phi_f$ corresponds to the function
  \[ f| _{w_p}|{\left(\begin{smallmatrix} 1/p &0 \\ 0 &1\end{smallmatrix}\right)} =\frac{g(\chi_p)}{a_p}\widetilde{f}| \left(\begin{smallmatrix} 1/p &0 \\ 0 &1\end{smallmatrix}\right) \]
  and hence we know its $q$-expansion. 
  
  In general, we need to know the action of the compositions of several such (commuting operators), and this is done similiarly as above, using the relations (see \cite{atkin_twists_1978})
  \[ f| \left(\begin{smallmatrix} 1/q &0 \\ 0 &1\end{smallmatrix}\right) | _{w_p} =\epsilon(q)f| _{w_p} | \left(\begin{smallmatrix} 1/q &0 \\ 0 &1\end{smallmatrix}\right)  \]

 Finally, we separate the isotypic components or the action of $N(\mathbb{F}_p)$ and thus obtain a $q$-expansion for a vector whose component at $p$ is now fixed for $K_{ns^+}(p)$, using the formula in the previous section.
  \item 
  Step 3: For primes $p$ where the label is "$ns^+$" but $f$ is supercuspidal at $p$ we the use in formula \ref{Cuspidal formula}. We see that we only have to separate the powers of $q$ according to their class mod $p$ to decompose into the isotypic parts for the local factor at $p$ of $f$.
\end{itemize}
\begin{example}
  We wish to find the $q$-expansion of a vector in $S_2(5s^+,7ns^+)$ with system of Hecke eigenvalues corresponding to (a Galois conjugate of) the newform $245h$ in $S_2(\Gamma_0(245))$, with notation in \cite{mfd}. This newform is not $7$-primitive, and the local representation at $7$ is ramified principal series $\mathrm{Ind}\, \omega_\chi \otimes\omega_\chi^{-1}$ such that $\omega_\chi:\mathbb{Z}_7^\times\to \mathbb{C}^\times$ is a character of exact order $3$. Hence we start with a newform $f\in S_2(\Gamma_1(35),\epsilon)$ with $\epsilon=\chi^2$ the cubic Dirichlet character mod $7$ that sends $3 \mapsto \zeta_3$. We first change $f$ to
  \[f+ \lambda_5 f| \left(\begin{smallmatrix} 1/5 &0 \\ 0 &1  \end{smallmatrix}\right) \]
  to make the component at $5$ fixed under the normalizer of the split Cartan, and then we need to compute the isotypic components for the unipotent radical $N\subset \GL_2(\mathbb{F}_7)$ of 
  \[ (f+ \lambda_5 f| \left(\begin{smallmatrix} 1/5 &0 \\ 0 &1  \end{smallmatrix}\right)) | _{w_7} | \left(\begin{smallmatrix} 1/7 &0 \\ 0 &1  \end{smallmatrix}\right)  \]
  \[= \lambda_7 \widetilde{f}| \left(\begin{smallmatrix} 1/7 &0 \\ 0 &1  \end{smallmatrix}\right)+ \lambda_7\lambda_5\epsilon(5)\widetilde{f}|  \left(\begin{smallmatrix} 1/35 &0 \\ 0 &1  \end{smallmatrix}\right)\]
  where 
  \[f| _{w_7}=\lambda_7\widetilde{f},\]
  \[\lambda_5=\frac{-1}{a_5}\] 
  \[\lambda_7=\frac{\sum_{\mathbb{F}_7^\times} \epsilon(x) e^{2\pi i x}}{a_7}\]
  
\end{example}
\begin{example}
  We want to find the $q$ expansion of a vector in $S_2(3ns^+,7s^+)$ with system of Hecke eigenvalues corresponding to the newform $63b$. This newform has a cuspidal local representation at $3$, whose corresponding parameter is the pair of characters of the unramified extension $\mathbb{Q}_3[s]/ (s^2 + 2s + 2 )$ with conductor 1 given by $s\mapsto d, 7 \mapsto 1$, where $d$ is a root of $d^2+1=0$. We first replacee $f$ with $f(z)-f(z/7)/7$ to obtain a vector that has correct level at $7$
  
  The form $f(z/3)$ correspond to a vector whose component at $3$ is fixed under the normalizer of the split Cartan open compact. We have
  \[ f(z/3)=\displaystyle\sum_{n>0} a_n t^n \]
  where $t=e^{2\pi i z/7}=q^{1/7}$. Thus if we write $f_i=\displaystyle\sum_{n=i \mathrm{ mod }7} a_n t^n$, then 
\[ f_i(z-1)=\zeta_7^{-i}f_i(z)  \]  
hence the $f_i$ are the $\psi^i$-isotypic part of $f(z/7)$ where $\psi:\mathbb{F}_7\to \mathbb{C}^\times$ is the character sending $1$ to $\zeta_7^{-1}$. Thus a vector fixed under $K(5b,7ns^+)$ in $\pi_f$ is given by 
\[ 
f_1-\frac{1}{7}\displaystyle \sum_{x\in\mathbb{F}_7^\times}\sum_{u\in \mathbb{F}_7}  \sum_{\mathrm{N}(t)=x(u^2+1)} \psi(u(x+1)-\mathrm{Tr}(t))\theta(t)
f_x  \]
\end{example}
  
\section{Gonality of modular curves}
\label{Gonality}
Fix a totally real field $F$. In view of Theorem \ref{Modularity lifting} and section \ref{Residual modularity}, an elliptic curve $E$ over $F$ is modular, unless 
$\overline{\rho}_{E,p}|_{G_{F(\zeta_p)}}$ is not absolutely irreducible for each $p=3,5,7$.
\begin{lem}
  If $G$ is a subgroup of $\GL_2(\mathbb{F}_p)$ which is not absolutely irreducible then $G$ is a subgroup of a Borel subgroup or a non-split torus.
\end{lem}
\begin{proof}
  Let $V$ be the underlying $\mathbb{F}_p$-vector space. We know $G$ preserves a line $L$ in $V\otimes \overline{\mathbb{F}}_p$. If $L$ is rational then $G$ acts reducibly on $V$ and $G$ is a subgroup of a Borel subgroup. If $L$ is not rational, it has a Galois conjugate distinct from it, which is also preserved by $G$. Thus $G$ is a subgroup of a torus.
\end{proof}
Thus the elliptic curves that we don't yet know to be modular gives rise to non-cuspidal $F(\zeta_{105})$-points on the modular curves $X(3*,5*,7*)$, where the label $*$ is either $b$ or $ns$, indicating a Borel level structure and a non-split Cartan level structure respectively. Recall that a curve is hyperelliptic (resp. bielliptic) over a field $k$ if it is a double cover of $\mathbb{P}^1$ (resp. an elliptic curve) over $k$.
\begin{lem}
  None of the modular curves $X(3*,5*,7*)$ above are hyperelliptic or bielliptic over $\mathbb{C}$.
\end{lem}
\begin{proof}
  We will make use of the following two facts
\begin{prop}(Castelnuovo-Severi inequality)
  Let $F$, $F_1$, $F_2$ be function fields of curves over a field $k$, of genera $g$, $g_1$, $g_2$, respectively. Suppose that $F_i\subseteq F$ and $F=F_1F_2$. Let $d_i=[F:F_i]$. Then
  \[ g\leq g_1d_1+g_2d_2 +(d_1-1)(d_2-1)\]
\end{prop}
\begin{proof}
  See \cite{stichtenoth_algebraic_1993}, III.10.3.
\end{proof}
\begin{thm}(Abramovich  \cite{abramovich_linear_1996})
Let $\Gamma\subset \mathrm{PSL}_2(\mathbb{Z})$ be a congruence subgroup of index $d$. Then the $\mathbb{C}$-gonality of the modular curve associated to $\Gamma$ is at least $\frac{7}{800}d$.
  \end{thm}
Recall that the gonality of a curve defined over a field $k$ is the smallest $d$ such that there exists a map from the curve to $\mathbb{P}^1$ defined over $k$ of degree $d$. Hyperelliptic and bielliptic curves have gonality $\leq 4$.
A non-split Cartan subgroup of $\GL_2(\mathbb{F}_p)$ has index $p(p-1)$ and a Borel subgroup has index $p+1$. Both groups contain the center and have surjects onto $\mathbb{F}_p^\times$. Consider the following cases:
\begin{itemize}
  \item $X(3*,5*,7*)$ where either $5*=5ns$ or $7*=7ns$:
  The index of the corresponding subgroup of $\mathrm{PSL}_2(\mathbb{Z})$ is at least $640$ or $1008$ respectively, and hence Abramovich's bound gives a $\mathbb{C}$-gonality $\geq 5$. Thus the lemma holds in this case.
  \item $X(3b,5b,7b)$:
  
  The space of cusp forms for $\Gamma_0(105)$ has dimension $13$. The subspace fixed by the Atkin-Lehner operator $w_{35}$ has dimension $3$, with the $q$-expansion of a basis given by
  \begin{align*}
f_1 &= q - q^2 - q^3 - q^4 + q^5 + q^6 - 7q^7 + 3q^8
    + q^9 - q^{10} - \\
    &\quad -4q^{11} + q^{12} - 2q^{13} + 7q^{14}
    - q^{15} - q^{16} + 2q^{17} - q^{18} + 4q^{19} + O(q^{20}),  \\
  f_2 &= q - q^2 + q^3 - q^4 + 3q^5 - q^6 - q^7 + 3q^8
    + q^9 - 3q^{10}\\  &\quad+ 4q^{11} - q^{12} - 2q^{13} + q^{14}
    + 3q^{15} - q^{16} - 6q^{17} - q^{18} + 4q^{19} + O(q^{20}),\\
  f_3 &= q + q^2 + q^3 - q^4 + q^5 + q^6 + q^7 - 3q^8
    + q^9 + q^{10} \\ &\quad- q^{12} - 6q^{13} + q^{14} + q^{15} - q^{16}
    + 2q^{17} + q^{18} - 8q^{19} + O(q^{20})
 \end{align*} 
  These cusp forms form a basis for $H^0(X_0(105)/w_{35},\Omega^1)$, and the $q$-expansion is the expansion in the formal neighborhood of the image of the cusp $\infty$. One checks by Magma that the above forms do not satisfy any quadratic relation, and thus the canonical map of $X_0(105)/w_{35}$ does not factor through a conic, hence it must be a quartic plane curve, and not hyperelliptic.
  Now suppose there exists a map $\pi: X_0(105)\to \mathbb{P}^1$ of degree $d\leq 4$. Of $\pi$ does not factor through the quotient map, the Castelnuovo-Severi inequality for $\pi$ and the quotient map to $X_0(105)/w_{35}$ would imply 
  
  \[ 13=g(X_0(105))\leq 0+2\times 3 +(d-1) \]
  which is a contradiction. Thus $\pi$ factor through the quotient map, and in particular $d=2$ or $4$. But that would imply $X_0(105)/w_{35}$ is either rational or hyperelliptic, a contradiction.
  
  \item  $X(3ns,5b,7b)$:
  
  If $X(3ns,5b,7b)$ were hyperelliptic or bielliptic, so is any curve dominated by it, by Proposition 1 of \cite{harris_bielliptic_1991}.\\ Thus it suffices to show the curve $X(3ns^+,5b,7b)$ is not hyperelliptic or bielliptic. Using lemma \ref{Jacobian decomposition}, we have up to isogeny
  \begin{align*}
     &\quad\ \, \Jac (X(3ns^+,5b,7b)) \times \Jac (X(3b,5b,7b)) \\
     &\sim \Jac (X(3s^+,5b,7b)) \times \Jac(X(5b,7b)) \\
     &\sim \Jac (X(9b,5b,7b)/w_9) \times \Jac(X(5b,7b)) \\
       \end{align*}
       and
   \begin{align*}     
     &\quad\ \,\Jac(X(3ns^+,5b)) \times \Jac(X(3b,5b)) \\
     &\sim \Jac(X(9b,5b)/w_9)
   \end{align*}
   The space of cusp forms for $\Gamma_0(315)$ has dimension $41$, and the subspace fixed by $w_9$ has dimension $21$. The space of cusp forms for $\Gamma_0(35)$ had dimension $3$, thus $X(3ns^+,5b,7b)$ has genus $11$. The $w_9$-fixed subspace of cusp forms for $\Gamma_0(45)$ has dimension $1$, and the space of cusp forms of $\Gamma_0(15)$ has dimension $1$, thus $X(3ns^+,5b)$ has genus $0$. Suppose there is a map $\pi: X(3ns^+,5b,7b) \to C$ of degree $2$, where $C$ has genus $g\leq 1$.  If the forgetting level structure at $7$ map $X(3ns^+,5b,7b) \to X(3ns^+,5b)$ (which has degree $8$) does not factor through $\pi$, the Castelnuovo-Severi inequality would imply
\[ 11 \leq 2g+0+7  \] 
which is a contradiction. But such a factorization would correspond to a subgroup of $\SL_2(\mathbb{Z})$ containing the congruence subgroup corresponding to $ X(3ns^+,5b,7b)$ with index $2$. However, the table in \cite{cummins_congruence_2004} shows no such groups exists.
  \end{itemize}

\end{proof}
\begin{thm}\label{Proof of Theorem A}
  There is a finite list of pairs $(j,F')$ where $F'/F$ is a totally real quadratic extension and $j\in F'$, such that an elliptic curve $E$ over any totally real quadratic extension of $F$ is modular unless $j(E)$ is in the list.
\end{thm}
\begin{proof}
  From what we have said, an elliptic curve $E$ over $F'$ will be modular unless it gives rise to a $F'(\zeta_{105})$-rational point on one of the modular curves $X=X(3*,5*,7*)$ above. Such a point is the same as a $F(\zeta_{105})$-rational effective divisor, that is a $F(\zeta_{105})$-rational point of $\mathrm{Sym}^2 X$. 
  By the above lemma, none of them are bielliptic or hyperelliptic, hence Collorary $3$ of \cite{harris_bielliptic_1991} applies and gives the desired finiteness.

\end{proof}
\begin{remark}
  The finiteness result in \cite{harris_bielliptic_1991} hinges on Faltings' theorem on subvarieties of abelian varieties, and thus the above theorem is ineffective for a general totally real field $F$. However in good cases (e.g. $F=\mathbb{Q}$), one can make the list computable, and we will attempt to do this with some other simplifying assumptions in the next section.
\end{remark}
\section{Modularity over real quadratic fields}
\label{Modularity over real quadratic fields}
By the previous section, we see that there are only finitely many pairs $(j,F)$ where $F$ is a real quadratic field and $j\in F$ is the $j$-invariant of an elliptic curve over $F$ that is not modular, namely the ones whose mod $p$ Galois representation have small image for all $p=$3, 5, 7. However this finiteness statement is ineffective due to the use of Falting's theorem. The goal of this section is to make the exceptional pairs explicit and proving modularity of the corresponding curves, under the simplifying assumption that $F$ is a totally real quadratic field unramified above $5$ and $7$. 

  Let $E$ be an elliptic curve over a totally real field $F$ such that $\sqrt{5}\notin F$. If $E$ were to be not modular, by theorem \ref{Modularity lifting} and section \ref{Residual modularity}, $\overline{\rho}_{E,p}|_{F(\zeta_p)}$ must be not absolutely irreducible for all $p\in \{ 3,5,7\}$, equivalently, the mod $p$ Galois representation becomes absolutely reducible over the quadratic subextension of $F(\zeta_p)/F$. This means that either $\overline{\rho}_p$ is absolutely reducible (hence reducible since it is odd), or absolutely irreducible but becomes absolutely reducible over $F(\sqrt{(-1)^{(p-1)/2}p})$ (because this is the unique quadaratic subextension of $F(\zeta_p)$ under our assumptions). In the latter case, $\overline{\rho}_p$ is the induction of a character from the Galois group of $F(\sqrt{(-1)^{(p-1)/2}p})$, and this character is valued either in $\mathbb{F}_p^\times$ or valued in $\mathbb{F}_{p^2}^\times$ but not in $\mathbb{F}_p^\times$. The above possibilities are reflected in terms the image of $\overline{\rho}$ as being conjugate to a subgroup of the Borel subgroup (reducible case), the normalizer of a split torus (irreducible but becomes reducible over $\mathbb{F}_p^\times$), or the normalizer of a non-split torus (irreducible, becomes irreducible but absolutely reducible over $\mathbb{F}_p^\times$) of $\GL_2(\mathbb{F}_p)$. Note that in the case $p=5$, the restriction of $\overline{\rho}_5$ to $F(\sqrt{5})$ is still odd, and hence this restriction will be absolutely irreducible if it is irreducible. We say that the elliptic curve $E$ has small image at $p$ for each $p$=3, 5, 7. Oberserve that the normalizer of a split torus in $\GL_2(\mathbb{F}_3)$ is a subgroup of index 2 in the normalizer of a non-split torus in $\GL_2(\mathbb{F}_3)$, as the latter are the 2-Sylow subgroups. Thus we only need to consider the Borel and normalizer of non-split Cartan level structures at 3, and Borel and normalizer of split Cartan level structure at 5.
  
We have the following observation over general totally real fields:
\begin{prop} \label{Automatic ordinarity prop}
 Let $F$ is any totally real field where 5, 7 are unramified, and $E$ is an elliptic curve defined over $F$ with small image at 3, 5, 7. Then $E$ is (nearly) ordinary at all places $v|5$ or is (nearly) ordinary at all places $v|7$. 
\end{prop}
\begin{proof}
   We first recall some facts about the type of a $p$-adic Galois representation. $E$ gives rise to a strictly compatible system of Galois representation $\rho_l$ defined over $\mathbb{Q}$, which in particular means that for each finite place $v$ of $F$, there exists a 2-dimensional Weil-Deligne representation $WD_v$ of $W_{F_v}$ with rational traces such that $WD_{F_v}$ is the Weil-Deligne representation associated to the Galois representations $\rho_l|_{G_{F_v}}$ via Grothendieck's $l$-adic monodromy theorem if $v\!\!\not| l$ or via a recipe of Fontaine if $v| l$. In the case $v\!\!\not| l$, if the monodromy operator $N=0$, then the Weil-Deligne and the Galois representation agree on the inertia subgroup $I_{F_v}$, and in particular is a representation defined over $\mathbb{Q}_l$. Note that as the compatible system has cyclotomic determinant, $WD_v|_{I_{F_v}}$ has trivial determinant. 
  \begin{lem}
    If $v$ is a place of $F$ above a prime $p>3$ then the inertial type $WD_v|_{I_{F_v}}\cong \phi\oplus \phi^{-1}$ where $\phi$ is a character of $I_{F_v}$ which has order dividing $4$ or $6$. 
  \end{lem}
  \begin{proof}
    We know that the inertia type is a finite image representation with trivial determinant. Because it also has a model over $\mathbb{Z}_2$, the size of the image can not be divisible by $p$, hence the representation factors through the tame quotient of $I_{F_v}$, which is pro-cyclic, with a topological generator $u$. The eigenvalues of $u$ must be $\zeta$, $\zeta^{-1}$ for some root of unity $\zeta$, and since the trace of $u$ is rational, this forces $\zeta$ to have order dividing 4 or 6.
    
Alternatively, one could work with (reduced) minimal Weiestrass equations to show that any elliptic curve over $F_v$ acquires semi-stable reduction over an extension with ramification index dividing 4 or 6, see \cite{silverman_advanced_1994}
  \end{proof}
Observe that the lemma shows that the image of the inertia in $WD_v$ must be a subgroup of $\mathrm{Im} \overline{\rho}_l\cap \SL_2(\mathbb{F}_l)$ if $l>3$ and $v\!\!\not| l$, since the kernel of the reduction map $\GL_2(\mathbb{Z}_l)\to \GL_2 (\mathbb{F}_l)$ is a pro-$l$ group. When $l=3$ and the inertial type has order divisible by $3$, the same statement still holds, because if $g\in GL_2(\mathbb{Z}_3)$ with $g^3=1$ then $g$ must reduce to a non-trivial unipotent element in $GL_2(\mathbb{F}_3)$ (because such $g$ gives an isomorphism of $\mathbb{Z}_3^2\cong \mathbb{Z}_3[\zeta_3]$ identifying $g$ with $\zeta_3$, and the mod $3$ reduction of multiplication by $\zeta_3$ is a non-trivial unipotent element).

  We split into the following cases:
\begin{itemize}
  \item 
 $E$ admits a Borel level structure at $3$ and either a Borel or normalizer of split Cartan level structure at 7. 
 
 For any place $v|7$ of $F$, the image of $WD_v|_{I_{F_v}}$ has order dividing 6. If $E$ has potential multiplicative reduction at $v$ then $E$ is potentially ordinary, hence is nearly ordinary at $v$. So we assume now that $E$ has potential good reduction at $v$.
  Suppose $E$ has minimal Weierstrass equation 
 \[ y^2= x^3+Ax+B\]
 over $\mathcal{O}_{F_v}$. Let $v(\Delta)=v(4A^3+27B^2)<12$ be the valuation of the minimal discriminant. The order of the image of $WD_v|_{I_{F_v}}$ is the degree of the smallest extension of $F_v^{nr}$ such that $E$ acquires good reduction \cite{serre_good_1968}. It is also the minimal $e$ such that $12|v(\Delta)e$. Since $E$ has potential good reduction, $v(A^3)\geq v(\Delta)$. Replacing $E$ with a quadratic twist, it suffices to consider the cases $e=3$ or $e=1$.

If $e=3$, $v(\Delta)$ is $4$ or $8$, and hence $v(A^3)\neq v(B^2)$, since otherwise $v(A^3)\geq v(\Delta)\geq v(A^3)=\mathrm{min}\{v(A^3),v(B^2)\}$ which forces $v(A^3) = v(\Delta)=v(B^2)$, a contradiction. Thus $v(\Delta)=\mathrm{min}\{v(A^3),v(B^2)\}$ is not divisible by $3$, so $v(A^3)>v(B^2)$, and thus $j(E)=0$ mod $v$. But this means $E$ has potential good ordinary reduction, hence $E$ is nearly ordinary at $v$.
 
 If $e=1$, $E$ has good reduction. Because $F_v$ is unramified over $\mathbb{Q}_7$ and $\rho_{E,7}|_{G_{F_v}}$ is crystalline with Hodge-Tate weight 0, 1, $(\overline{\rho}_{E,7}|_{I_{F_v}})^{ss}\cong \omega_2\oplus \omega_2^7$ or $\cong \omega_1\oplus 1$, where $\omega_n$ is the tame character of niveau $n$ of $\mathrm{Gal}(\overline{\mathbb{Q}}_7/\mathbb{Q}_7^{ur})$. If first case occur, the image of $\overline{\rho}_{E,7}$ contains an element which has non-zero trace and irreducible characteristic polynomial, hence can not be a subgroup of the Borel or normalizer of split Cartan subgroup. Hence the second case occur, which means that $E$ has good ordinary reduction, and hence is ordinary. 
 
 \item $E$ admits a normalizer of non-split Cartan level structure at 3, or a normalizer of non-split Cartan level structure at 7.
 
 For any place $v|5$ of $F$, the image of $WD_v|_{I_{F_v}}$ must be a 2-group, and hence has order dividing 4. Let $e$ denote its order, as above. We work with the minimal Weierstrass equation of $E$ as in the previous case. As before, we can assume $E$ has potential good reduction. Replacing $E$ with a quadratic twist, we can assume $e=4$ or $e=1$.
 
 If $e=4$, $v(\Delta)$ is 3 or 9,  and hence $v(A^3)\neq v(B^2)$, since otherwise $v(A^3)\geq v(\Delta)\geq v(A^3)=\mathrm{min}\{v(A^3),v(B^2)\}$ which forces $v(A^3) = v(\Delta)=v(B^2)$, a contradiction. Thus $v(\Delta)=\mathrm{min}\{v(A^3),v(B^2)\}$ is not divisible by $2$, so $v(A^3)<v(B^2)$, and thus $j(E)=1728$ mod $v$. But this means $E$ has potential good ordinary reduction, hence $E$ is nearly ordinary at $v$.
 
If $e=1$, $E$ has good reduction at $v$. By exactly the same argument as in the previous case, the fact that $E$ admits either a Borel or normalizer of split Cartan level structure at 5 forces $E$ to have good ordinary reduction, hence $E$ is nearly ordinary at $v$.
 
 \end{itemize}
\end{proof}
\begin{remark} \label{Automatic ordinarity}
   Being nearly ordinary at all places $v$ above a prime is the important crucial local condition to apply the modularity lifting theorems with small residual images of Skinner-Wiles \cite{skinner_residually_1999}, \cite{skinner_nearly_2001}. Under our assumptions, their modularity lifting theorems for irreducible residual representations apply. Unfortunately the very restrictive conditions required in the residually reducible case (namely, that the splitting field of the ratio of the characters occurring in the residual representation is required to be abelian over $\mathbb{Q}$. This is only an issue for $p>3$) prevents us from fully exploiting the above proposition.
\end{remark}
\begin{prop}
\label{Exceptional curves}
  Let $F$ be a totally real quadratic field where 5 and 7 are unramified, and $E$ is an elliptic curve over $F$. Then $E$ is modular unless $j(E)$ is the $j$-invariant of a degree at most 2 point on one of the following curves
  \begin{itemize}
  \item 
   $X(3b,5s^+)$
   
   
  \item 
  $X(3ns^+,7s^+)$
  
  \item 
   $X(5b,7b)$
  
   \item
    $X(5b,7ns^+)$
  \end{itemize}
\end{prop}
\begin{proof}
  We already know that $E$ is modular unless it has small image at all primes $p$=3, 5, 7. There are 12 possible combination of level structures at 3, 5, 7, and hence $E$ is modular unless it has the same $j$-invariant as an elliptic curve that comes from an $F$-point of $X(3*,5*,7*)$ where the choice of the level structure *=$b$ or $ns^+$ at 3; *=$b$ or $s^+$ at 5; and *=$b$, $s^+$ or $ns^+$ at 7. Hence either $E$ gives rise to a quadratic point on one of the four curve listed, or on one of the curves $X(3b,5b,7s^+)$, $X(3ns^+,5s^+,7b)$, $X(3ns^+,5s^+,7ns^+)$. By remark \ref{Automatic ordinarity}, we can apply the main theorem of \cite{skinner_nearly_2001} for the prime 5 for the last two curves, and for the prime 7 for the first curve. We note that the mistake in \cite{skinner_nearly_2001} does not cause any problem here, since we assumed 5, 7 are unramified, and that the distinguishedness condition in loc.cit. is unnecessary if one works with framed deformation rings.
  \end{proof}

\begin{thm}\label{Proof of Theorem B}
  Suppose $F$ is a totally quadratic field such that 5 and 7 are unramified in $F$. Then every elliptic curve over $F$ is modular.
\end{thm}
\begin{proof}
  This follows from Proposition \ref{Exceptional curves} and the study of quadratic points on some modular curves in section \ref{Quadratic Points} below.
\end{proof}
\begin{remark}
\label{General quadratic fields}
  As the proof of proposition \ref{Exceptional curves} shows, to get modularity for a real quadratic field different from $\mathbb{Q}(\sqrt{5})$, we only need to study quadratic points on the four curves listed there and $X(3b,5b,7s^+)$, $X(3ns^+,5s^+,7b)$, $X(3ns^+,5s^+,7ns^+)$. We can further reduce to understanding quadratic points on the curves $X(3b,7s^+)$, $X(3ns^+,7b)$ and $X(5s^+,7ns^+)$. 
  \begin{itemize}
    \item
The curve $X(3b,7s^+)$ has genus 6, its Jacobian decomposes up to isogeny as
\begin{align*}
  \mathrm{Jac}(X(3b,7s^+))\sim E_1\times A_1\times A_2 \times E_2
\end{align*}
where the first three factors have conductor 147 while the last one has conductor 21, the factors $E_i$ are elliptic curves while the factors $A_i$ are abelian surfaces. All factors except for $A_1$ has rank 0 over $\mathbb{Q}$. An approach similar to the one used to handle the curve $X(3ns^+,7s^+)$ below allows one to explicitly write down maps from $X(3b,7s^+)$ to the elliptic curves $E_1$, $E_2$, and hence find its quadratic points by the same method.
    \item 
    The curve $X(3ns^+,7b)$ has genus 2, and its Jacobian has rank 0 and the hyperelliptic involution given by the Atkin-Lehner involution $w_7$. Using the same method for the curve $X(5b,7b)$ below, we can determine all the quadratic points on it that does not come from the hyperelliptic class, while the points coming from the hyperelliptic class only gives rise $j$-invariants of $\mathbb{Q}$-curves, and hence the corresponding elliptic curves are modular by lemma \ref{Q-curve}.
    \item  
    The curve $X(5s^+,7ns^+)$ has genus 19, and its Jacobian admits two abelian surface factors that has rank 0 over $\mathbb{Q}$. Thus it is in theory possible to determine all quadratic points on it. However due to practical (computational) complications in executing this, we leave this to a future work. 
  \end{itemize}
  In particular, using the methods in this paper, for modularity of elliptic curves over all real quadratic fields different from $\mathbb{Q}(\sqrt{5})$, the only curve we can not directly handle is the genus 19 curve $X(5s^+,7ns^+)$. However, proposition \ref{Automatic ordinarity prop} shows that an elliptic curve corresponding to a quadratic point defined over a field unramified at 5 on this curve is ordinary at all places above 5, and remark \cite{skinner_nearly_2001} shows such curves a modular. Thus the methods of this paper can actually show that all elliptic curves over a real quadratic field unramified at 5 are modular (that is, we do not need the field to be unramified at 7).
\end{remark}
\section{Quadratic points on modular curves}\label{Quadratic Points}
The goal of this section is to show that any elliptic curve that gives rise to a quadratic point on one of the modular curves in Proposition \ref{Exceptional curves} are modular. For each such modular curve $X$, at each prime $p$ such that $X$ has a Borel level structure at $p$, there is an Atkin-Lehner involution $w_p$ which is an involution of $X$ over $\mathbb{Q}$, which in the moduli interpretation of $X$ correspond to
\[ (E,\phi) \mapsto (E/C_p,\phi')\]
where $C_p$ is the line that defines the Borel subgroup in the level structure at $p$. The Atkin-Lehner involutions generate an elementary abelian $2-$subgroup of \\$\mathrm{Aut}(X/\mathbb{Q})$. We call any non-trivial element of this subgroup an Atkin-Lehner involution.
We have the following useful fact
\begin{lem}
\label{Q-curve}
  Let $E$ is an elliptic curve over a quadratic field $F$ that gives rise to a point $P\in X(F)$. Assume that there is an Atkin-Lehner involution $w$ such that $P$ maps to a rational point in $X/w$. Then $E$ is modular
  
\end{lem}
\begin{proof}
We only need to consider the case that $E$ has no CM.

  Let $\sigma\in \mathrm{Gal}(F/\mathbb{Q})$ denote the non-trivial element. Then $E^\sigma$ gives rise to the point $P^\sigma\in X(F)$, and $P^\sigma=w(P)$ or $P^\sigma=P$. In either case, a quadratic twist of $E^\sigma$ must be $F$-isogenous to $E$. Thus there is a quadratic character $\chi$ of $G_F$ such that over $\mathbb{Q}_l$
  \[\rho_{E,l}\cong \rho_{E,l}^\sigma \otimes \chi\]
  Thus for any $\tau\in G_\mathbb{Q}$, $E^\tau$ is isogenous to $E$ over $\overline{Q}$, and the isogeny is defined over $F(\chi)$, the splitting field of $\chi$.
  Thus $E$ is a $\mathbb{Q}$-curve in the sense of \cite{ribet_abelian_2004}. By the proof of theorem 6.1 of \cite{ribet_abelian_2004}, $E$ occurs as an $F(\chi)$-factor of some abelian variety $A$ of $\GL_2$-type over $\mathbb{Q}$. Serre's conjecture (proven in \cite{khare_serres_2009-1} and \cite{khare_serres_2009}) implies that $A$ is a $\mathbb{Q}$-factor of some modular Jacobian $\mathrm{Jac}(X_0(N))$, this shows that $\rho_{E,l}|_{G_{F(\chi)}}$ is automorphic (in fact corresponding to a base change of a classical modular form). By solvable base change, $\rho_{E,l}$ is modular, and hence $E$ is modular.
  
\end{proof}
\subsection{The curve $X(3b,5s^+)$}
We have $X=X(3b,5s^+)\cong X_0(75)/w_{5^2}$, hence $\Jac(X(3b,5s^+))\sim E_1\times E_2 \times E_3 $ in the isogeny category. Here $E_1$ is isogenous to $X_0(15)$ while $E_2$ and $E_3$ are elliptic curves of conductor $75$, and each of them have rank $0$ over $\mathbb{Q}$.

The $q$-expansion of the three newforms corresponding to $E_i$ are
 \begin{align*}
 &f_1=q - q^2 - q^3 - q^4 + q^5 + q^6 + 3q^8 + q^9 - q^{10} - 4q^{11}+ O(q^{12}) \\
 &f_2=q + q^2 + q^3 - q^4 + q^6 - 3q^8 + q^9 - 4q^{11} + O(q^{12}) \\
 &f_3=q - 2q^2 + q^3 + 2q^4 - 2q^6 + 3q^7 + q^9 + 2q^{11} + O(q^{12})
 \end{align*}
Using the method desribed in section \ref{q-expansion} we find that the $q$-expansion of a basis for \\$H^0(X,\Omega^1)$ is given by $-5f_1(z)+f_1(z/5)$, $f_2(z/5)$ and $f_3(z/5)$. Using this basis, we see there are no degree 2 relations between them and there is a degree 4 relation, hence the canonical map realizes $X$ as the quartic
\[9X^4 + 30X^2Y^2 + 108X^2YZ - 48X^2Z^2 + 25Y^4 - 60Y^3Z - 80Y^2Z^2 +
    16Z^4\]
and thus $X$ is not hyperelliptic. Over $\mathbb{Q}$, the automorphism group of $X$ has order 2, generated by the Atkin-Lehner involution $w_3$, which in the above model correspond to $[X:Y:Z]\mapsto [-X:Y:Z]$. Using Magma we find that the quotient curve is the elliptic curve $E_1$ with equation
\begin{align*}
  y^2 + xy + y = x^3 + x^2 - 5x + 2
\end{align*}
and the quotient map $\phi_1:X\to E_1$ is given in terms of homogenous coordinates by
\begin{align*}
  [&-9/4X^2Y^2 - 15/2Y^4 + 9/20X^2YZ - 51/2Y^3Z + 9/50X^2Z^2 + 18Y^2Z^2
    - 6/25YZ^3 - 24/25Z^4:\\
&45/16X^2Y^2 + 135/16Y^4 + 9/10X^2YZ + 39Y^3Z - 81/100X^2Z^2 -
    39/10Y^2Z^2 - 363/25YZ^3 + 93/25Z^4:\\&-9/4X^2Y^2 - 15/2Y^4 + 9/5X^2YZ - 21Y^3Z - 9/25X^2Z^2 + 162/5Y^2Z^2
    - 348/25YZ^3 + 48/25Z^4]
\end{align*}
We have $E_1(\mathbb{Q})\simeq \mathbb{Z}/2\mathbb{Z}\oplus \mathbb{Z}/8\mathbb{Z}$.
Over $\mathbb{Q}(\sqrt{5})$, the automorphism group of $X$ is isomorphic to $S_3$, and for the automorphism of order $3$ given by 
\begin{align*}
  [X:Y:Z]\mapsto [-1/2X+\sqrt{5}/2Y:-3\sqrt{5}/10X-1/2Y:Z]
\end{align*}
the quotient curve is the elliptic curve $E_2$ with equation
\begin{align*}
  y^2 + y = x^3 + x^2 + 2x + 4
\end{align*}
and the quotient map $\phi_2:X\to E_2$ is given by
\begin{align*}
  [&-12/5X^3Y^4 - 36/5XY^6 - 357/250X^3Y^3Z - 1719/50XY^5Z +
    27/25X^3Y^2Z^2 - 687/125XY^4Z^2 \\ &+402/625X^3YZ^3 +
    2268/125XY^3Z^3 + 12/125X^3Z^4 + 864/625XY^2Z^4 - 984/625XYZ^5 \\&-
    144/625XZ^6:-3/4X^3Y^4 + 81/100X^2Y^5 - 9/4XY^6 + 63/20Y^7 - 12/25X^3Y^3Z \\&-
    81/50X^2Y^4Z - 54/5XY^5Z +117/50Y^6Z - 81/25X^2Y^3Z^2 -
    84/25XY^4Z^2 \\&- 891/25Y^5Z^2 + 48/625X^3YZ^3 - 162/125X^2Y^2Z^3 +
    144/125XY^3Z^3 - 702/25Y^4Z^3 \\& + 12/625X^3Z^4 + 396/625XY^2Z^4 +
    1188/125Y^3Z^4 + 48/625XYZ^5 + 216/25Y^2Z^5 - 144/125YZ^6 \\&-
    288/625Z^7:3/2X^3Y^4 + 9/2XY^6 + 24/25X^3Y^3Z + 108/5XY^5Z + 168/25XY^4Z^2 -
    96/625X^3YZ^3 \\&- 288/125XY^3Z^3 - 24/625X^3Z^4 - 792/625XY^2Z^4 -
    96/625XYZ^5]
\end{align*}
We have $E_2(\mathbb{Q})\simeq \mathbb{Z}/5\mathbb{Z}$ is cyclic of order 5, generated by the point $[-1:1:1]$.
If $P$, $P^\sigma$ is a pair of conjugate quadratic points, then $\phi_i(P)+\phi_i(P^\sigma)$ is a rational torsion point on $E_i$, thus we have $\phi_2(P)-a[-1:1:1]=-(\phi_2(P^\sigma)+a[-1:1:1])$ for some integer $a$ mod 5, while $2\phi_1(P)-b[0:1:1]=-(2\phi_1(P^\sigma)-b[0:1:1])$ for $b=0$ or 1. The two equality implies $\phi_i(P)$, $\phi_i(P^\sigma)$ have the same image under a suitable two-to-one map $E_i\to \mathbb{P}^1$, thus $P$, $P^\sigma$ have the same image under a map $C\to \mathbb{P}^1\times \mathbb{P}^1$, where the two coordinate map have degree 6 and 16. Depending on the value of $a$, $b$, this maps $X$ birationally onto its image or maps $X/w_3$ birationally onto its image. Thus, either $P$ and $P^\sigma$ map to the same point in $X/w_3$, or they map to the same singular point (which is necessarily defined over $\mathbb{Q}$) in the image of the map. Using Magma, under a birational isomorphism $\mathbb{P}^1\times \mathbb{P}^1\simeq \mathbb{P}^2$, we find the plane curve which is the image of $X$, and find its singular points over $\mathbb{Q}$. The resulting quadratic points that we get either satisfies $w_3P=P^\sigma$ (hence correspond to $\mathbb{Q}$-curves), are CM or is defined over a real quadratic field with 5 ramified, except for 2 conjugate pair of points defined over $\mathbb{Q}(\sqrt{41})$. For the last 2 conjugate pair of points, we check directly that the $j$-invariant is not in the image of a $\mathbb{Q}(\sqrt{41})$-point of $X(7b)$, $X(7s^+)$ or $X(7ns^+)$, so that the image of the mod 7 representation is large and hence the points are modular. Thus all points defined over quadratic fields where 5 is unramified gives rise to modular elliptic curves.  
\subsection{The curve $X(5b,7b)$}
The curve $X=X(5b,7b)$ has Jacobian $\Jac(X)\sim E \times A$ where $E$ is an elliptic curve and $A$ is an abelian surface of conductor $35$. Both $E$ and $A$ has rank $0$ over $\mathbb{Q}$. $A$ corresponds to a pair of conjugate newforms with coefficient field $\mathbb{Q}[x]/(x^2+x-4)$. The pair of Atkin Lehner involutions ($w_3$, $w_5$) has sign $(1,-1)$ and $(-1,1)$ on $E$ and $A$ respectively. It follows that $w_{35}=w_5 w_7$ is the hyperelliptic involution on $X$, as the quotient $X/w_{35}$ has genus 0. 
The $q$-expansion of a basis for $H^0(X,\Omega^1)$ is given by
\begin{align*}
&f_1=q + q^3 - 2q^4 - q^5 + q^7 - 2q^9 - 3q^{11} +O(q^{12})\\
&f_2=2q - q^2 - q^3 + 5q^4 + 2q^5 - 8q^6 - 2q^7 - 9q^8 + 3q^9 - q^{10} + q^{11}+O(q^{12})\\
&f_3=q^2 - q^3 - q^4 + q^8 + q^9 + q^{10} + q^{11} +O(q^{12})
\end{align*}
where the $f_1$ corresponds to $E$ and $f_2$, $f_3$ corresponds to $A$.
The canonical map is given by $X\to X/w_{35} \hookrightarrow \mathbb{P}^2$ as a double cover of the conic
\[-4X^2 + Y^2 + 2YZ + 17Z^2
\]
The quotient $X/w_7$ is a genus $2$ curve with Jacobian isogenous to $A$. Putting $x=f_3/f_2$, $y=4dx/(f_2dq/q)$, an equation for this curve is given by
\[y^2=-7599x^6 - 3682x^5 - 1217x^4 - 284x^3 - 17x^2 - 2x + 1\]
The group of rational points of $\Jac(X/w_7)$ has order $16$, and the rational degree $2$ divisors that are not the hyperelliptic class is given by in Mumford's notation
\begin{align*}
  &(x^2 + 7/50x + 3/50, 701/2500x - 121/2500),\, (x^2, -x + 1),\\
&(x^2 + 5/58x + 3/58, -3345/3364x - 905/3364),\, (x^2 + 4/19x + 1/19, -776/361x + 72/361)\\
&(x^2 + 1/8x + 1/8, -55/64x + 145/64),\, (x^2 + 2/15x + 1/15, 2/75x - 14/75)\\
&(x^2 + 1/3x, -3x - 1),\, (x^2 + 2/17x + 1/17, 0).
\end{align*}
or their images under the hyperelliptic involution.

Thus if $P\in X(F)$ is a quadratic point, then $P+P^\sigma$ must become one of the above $15$ divisors, or becomes the hyperelliptic class in $X/w_7$. But in the latter case, because the hyperelliptic involution on $X/w_7$ is induced by $w_5$, this means that there is an Atkin-Lehner involution on $X$ such that $wP=P^\sigma$, hence all such points must come from a modular elliptic curve by lemma \ref{Q-curve}. Since we are only interested in quadratic points defined over totally real fields, we only need to consider the cases where the image of $P+P^\sigma$ is $2(0,1)$, $2(0,-1)$, $(0,-1/3)+(-1/3,0)$ or $(0,1/3)+(-1/3,0)$. However since $X\to X/w_7$ is a double cover, the second case can not happen since the fiber of a rational point on $X/w_7$ is stable under the Galois action, hence if $P$ occurs in a fiber then $P^\sigma$ occurs in the same fiber. Thus the only case left is when $P$, $P^\sigma$ are in the fiber of $(0,1)$ or $(0,-1)$, but in that case $P=w_7 P^\sigma$ and hence the corresponding elliptic curve is modular, again by lemma \ref{Q-curve}.

\subsection{The curve $X(3ns^+,7s^+)$}
We compute an equation for $X=X(3ns^+,7s^+)$ by a method due to Noam Elkies (private communication), which is reproduced below. The modular curve $X_0(49)$ is isomorphic to the elliptic curve
\begin{align*}
  y^2+xy=x^3-x^2-2x-1
\end{align*}
The only rational points on $X_0(49)$ are the origin $O$ and the 2-torsion poin $T=[2:-1:1]$. Under a suitable identification, $O$ and $T$ are the two cusps and the Atkin-Lehner involution $w_{49}$ must be $P\mapsto T-P$, since it acts as $-1$ on the space of holomorphic differentials and swaps the cusps. The quotient of $X_0(49)$ by $w_7$ is the genus 0 curve with coordinate $h=(1+y)/(2-x)$. The $q$-expansion of $h$ can be computed from the modular parametrization of $X_0(49)$, and gives
\begin{align*}
  h=q^{-1}+2q+q^2+2q^3+3q^4+4q^5+5q^6+7q^7+8q^8+\cdots
\end{align*}
Writing $j(q^7)$ as a rational function of degree $28$ of $h$ by solving a linear system of equations in the coefficients we have
\begin{align*}
  j(q^7)=\frac{(h+2)  ((h+3)(h^2-h-5)(h^2-h+2)(h^4+3h^3+2h^2-3h+1))^3}
     { (h^3+2h^2-h-1)^7 }
\end{align*}
One the other hand the curve $X(3ns^+)$ is the cyclic triple cover of the $j$-line obtained by adjoining $j^{1/3}$. Hence the curve $X$ is a cyclic triple cover of the $h$-line, obtained by adjoining a cube root of $(h^3+2h^2-h-1)/(h+2)$. This gives the following quartic model for $X$
\begin{align*}
  (h+2)g^3=(h^3+2h^2-h-1).
\end{align*}
Using lemma \ref{Jacobian decomposition}, we compute up to isogeny over $\mathbb{Q}$ 
\begin{align*}
  \Jac(X)\sim E\times A
\end{align*}
where $E$ is an elliptic curve of conductor 441 (and has rank 1) while $A$ is an abelian surface of rank 0 and conductor 63. To determine the quadratic points on $X$, we wish to compute a model for $A$ and a map $X\to A$. The abelian surface $A$ is (isogenous to) the Weil restriction of a $\mathbb{Q}$-curve $E$ defined over $K=\mathbb{Q}(\zeta_3)$. There is a map $X_0(441)\to X$ via the identification $X_0(441)=X(3s,7s)$ and the map is obtained by containment of the corresponding congruence subgroups. Thus it suffices to write down a parametrization of $E$ by $X_0(441)$ which factors through $X$, which we now describe.

Let $f_1$, $f_2$ be an integral basis of a Hecke-stable two dimensional subspace of $H^0(X_0(63),\Omega^1)$ on which the Hecke operators act through the system of Hecke eigenvalues correspond to $A$. We normalize this choice by requiring the $q$-expansion $f_1=q+\cdots$ and $f_2=q^2+\cdots$. Let $\pi_1$, $\pi_2$ denote the two degeneracy maps $X_0(441)\to X_0(63)$ (where $\pi_1$ is the quotient map from the inclusion of congruence subgroups). Putting
\begin{align*}
  \Omega= (\pi_1^*-\pi_2^*)(f_1+(2-\zeta_3^{-1})f_2)
\end{align*}
we compute the integration map $\int_{i\infty} \Omega: X_0(441)\to \mathbb{C}$. Up to high precision, the image of the homology of $X_0(441)$ is a lattice $\Lambda$ with
\begin{align*}
  &g_2(\Lambda)=\frac{7\sqrt{-3}-41}{6144} \\
  &g_3(\Lambda)=\frac{42\sqrt{-3}-43}{884736}
\end{align*}
Suppose that $D\int_{i\infty} \Omega: X_0(441) \to \mathbb{C}/\Lambda= \{y^2=4x^3-g_2x-g_3\}$ factors through $X_0(441)\to X$, for some integer $D$. This is equivalent to a map $X\to \{y^2=4x^3-g_2D^4x-g_3D^6\}$ such that the pullback of $dz$ is $\Omega$. The coordinates $x$, $y$ of such a map must satisfy the differential system
\begin{align*}
  &y^2=4x^3-g_2D^4x-g_3D^6 \\
  &\quad\qquad dx/y=\Omega
\end{align*}
This system has a unique solution with $x=q^{-2}+\cdots$ in the ring of Laurent series $K((q))$, but only for suitable choice of $D$ will the formal solution lie in the function field of $X$. Note that there is an automorphism of $X$ of order $3$, defined over $K$ given by $g\to \zeta_3 g$. This automorphism fixes the cusps $\infty$, and its action on the formal neighborhood of $\infty$ is given by $q\to \zeta_3 q$ (we note that both $q$ and $g/h$ are uniformizers at $\infty$). Hence given a formal solution $(x,y)$ to the above differential system, we can recognize whether $(x,y)$ lives in the function field of $X$ by separating the $q$-expansion into 3 pieces according to the exponent of $q$ mod 3, and testing whether each piece is a rational function of $h$ times $g^i$.

Using this procedure, we found that for $D=12$, the formal solution is actually in the function field of $X$, and subsequent direct algebraic manipulation verifies that we indeed have a map of curves $\phi:X\to E$ defined over $K$ given by those functions. The group $E(K)$ is trivial, hence for any pair of conjugate quadratic points $P$, $P^\tau$ satisfies $\phi(P)=-\phi(P^\sigma)$, in particular they map to the same point after composing $\phi$ with the $x$-coordinate map $E\to \mathbb{P}^1$. Our computation shows that this composition map is of the form
\begin{align*}
  P_0(h)^2+P_1(h)g+P_2(h)g^2
\end{align*}
where $P_i(h)$ are rational functions in $h$ of degree 13, 29, 29. Note that the same argument applies to the maps $\phi\circ c$ and $\phi \circ c^2$, where $c$ is the automorphism $(g,h)\mapsto (\zeta_3 g,h)$, and also when we replace $\phi$ with the map $\phi^\sigma$, where $\sigma$ is the non-trivial automorphism of $K$. But this implies that the all three functions $P_0(h)^2$, $P_1(h)g$, $P_2(h)g^2$ and their $\sigma$-conjugates take the same values at $P$ and $P^\sigma$, because the system of linear equations
\begin{align*}
  &(P_0(h)^2-P_0(h^\tau)^2)+(P_1(h)g-P_1(h^\tau)g^\tau)+(P_2(h)g^2-P_2(h^\tau)(g^\tau)^2)=0 \\
  &(P_0(h)^2-P_0(h^\tau)^2)+(P_1(h)g-P_1(h^\tau)g^\tau)\zeta_3+(P_2(h)g^2-P_2(h^\tau)(g^\tau)^2)\zeta_3^2=0 \\
  &(P_0(h)^2-P_0(h^\tau)^2)+(P_1(h)g-P_1(h^\tau)g^\tau)\zeta_3^2+(P_2(h)g^2-P_2(h^\tau)(g^\tau)^2)\zeta_3=0 
\end{align*}
only has trivial solution.
This forces the $h$-coordinate of $P$ to be a zero of a suitable resultant, from which we easily get the list of possible $h$-coordinates of a $P$. We end up with the following list of quadratic points $[1:h:g]$
\begin{align*}
  &[0:1:0], [0:0:1], [1:-1:1],[1:\frac{-1+\sqrt{5}}{2}:\frac{1-\sqrt{5}}{2}], [1:\frac{-3-\sqrt{5}}{2}:\frac{1-\sqrt{5}}{2}],\\&[1:\frac{-1+\sqrt{13}}{2}:1], [1:\sqrt{5}:\frac{1+\sqrt{5}}{2}],[1:\frac{3+\sqrt{17}}{2}:\frac{5+\sqrt{17}}{4}]. 
\end{align*}
From the formula for $j$ in terms of $h$, we can check that all the above points gives rise to cusps or CM $j$-invariants, hence the corresponding elliptic curves are modular.

\subsection{The curve $X(5b,7ns^+)$}
Throughout this section we use the abbreviation $X=X(5b,7ns^+)$, $\alpha=\frac{-1+\sqrt{-7}}{2}$, $K=\mathbb{Q}(\sqrt{-7})$ and $\sigma$ the non-trivial automorphism of $K/\mathbb{Q}$. First we recall that the modular curve $X(5b)=X_0(5)$ is isomorphic to $\mathbb{P}^1$ over $\mathbb{Q}$, and an explicit rational coordinate $x$ such that 
\[ j=\frac{(x^2+10x+5)^3}{x}, \]
see \cite{elkies_elliptic_1998}. The Atkin-Lehner involution on $X_0(5)=X(5b)$ in terms of this coordinate is given by 
$ x\mapsto 125/x$.
The modular curve $X(7ns^+)$ parameterizing normalizer of non-split Cartan level structure at $7$ is also isomorphic to $\mathbb{P}^1$ over $\mathbb{Q}$, with a rational coordinate $\phi$ such that
\[j= 64\frac{(\phi(\phi^2+7)(\phi^2-7\phi+14)(5\phi^2-14\phi-7))^3} {(\phi^3-7\phi^2+7\phi+7)^7}.\]
The normalizer of non-split Cartain subgroup of $\mathrm{PSL}_2(\mathbb{F}_7)$ is not maximal, but is contained in a subgroup of order $24$ isomorphic to $S_4$. All such subgroups are conjugate under $\mathrm{PGL}_2(\mathbb{F}_7)$, but breaks up into two conjugacy class in $\mathrm{PSL}_2(\mathbb{F}_7)$. A choice of this conjugacy class gives a modular curve that parameterizes an ''$S_4$'' level structure at $7$ is defined over $\mathbb{Q}(\sqrt{-7})$, and has coordinate $\psi$
such that
\begin{align*}
  \psi&=\frac{(2+3\alpha)\phi^3-(18+15\alpha)\phi^2+(42+21\alpha)\phi+(14+7\alpha)}{\phi^3-7\phi^2+7\phi+7}\\ 
j&= (\psi-3(1+\alpha))(\psi-(2+\alpha))^3(\psi+3+2\alpha)^3
\end{align*}
(We refer the reader to \cite{elkies_klein_1999} for these facts).
Thus $X$ has is birational to the plane curve given by
\[\frac{(x^2+10x+5)^3}{x}=64\frac{(\phi(\phi^2+7)(\phi^2-7\phi+14)(5\phi^2-14\phi-7))^3} {(\phi^3-7\phi^2+7\phi+7)^7} \]
and if we let $Y$ denote the modular curve with Borel level structure at $5$ and ''$S_4$'' level structure at $7$, $Y$ has birational model
\begin{align*}
  \frac{(x^2+10x+5)^3}{x}=(\psi-3(1+\alpha))(\psi-(2+\alpha))^3(\psi+3+2\alpha)^3
  \end{align*}
We have a map $\pi: X \to Y$ given by 
\begin{align*}
    (x,\phi) \mapsto  (x,\frac{(2+3\alpha)\phi^3-(18+15\alpha)\phi^2+(42+21\alpha)\phi+(14+7\alpha)}{\phi^3-7\phi^2+7\phi+7})
    \end{align*}
    and its conjugate $\pi^\sigma:X\to Y^\sigma$.
    
Using lemma \ref{Jacobian decomposition}, we have up to isogeny over $\mathbb{Q}$ 
\[ Jac(X)\simeq A_1 \times A_2 \times A_3\]
where $A_i$ are the abelian surface factors of $J_0(245)^{new}$ on which $w_7$ acts trivially, and the action of $w_5$ is 1,-1,-1 respectively. Checking for inner twists of the newforms contributing to $X$, we see that the$A_i$ are absolutely simple, are non-isogenous over $\mathbb{Q}$, but $A_1$ is isogenous to $A_2$ over $K$. $A_3$ is not isogenous to $A_1$ even over $\mathbb{C}$. The factors $A_2$, $A_3$ have rank $0$ over $\mathbb{Q}$, and the order of the group $A_2(\mathbb{Q})$ divides $7$. The Hecke field of $A_1$, $A_2$ are $\mathbb{Q}(\sqrt{2})$ (These facts can be extracted from the tables in \cite{mfd}, the assertion about the rank follows from numerically computing the value at $s=1$ of the $L$-function)

Let us now consider the three open compact subgroups $G_1$, $G_2$, $H$ of $\GL_2(\widehat{\mathbb{Z}})$ given by the following local conditions 
\begin{itemize}
  \item The component at $p\!\!\not| 35$ is $GL_2(\mathbb{Z}_p)$
  \item The component at $5$ is the inverse image of the upper triangular matrices under the reduction map $\GL_2(\mathbb{Z}_5)\to \GL_2(\mathbb{F}_5)$
  \item The component at $7$ of $G_1$ is the subgroup of  $\GL_2(\mathbb{Z}_7)$ that reduces to the normalizer of a non-split Cartan subgroup of $\GL_2(\mathbb{F}_7)$
  \item 
  The component at $7$ of $G_2$ is the subgroup of  $\GL_2(\mathbb{Z}_7)$ that reduces to the normalizer of the subgroup of the non-split Cartan subgroup of $\GL_2(\mathbb{F}_7)$ whose determinant is a square in $\mathbb{F}_7$.
  \item 
  The component at $7$ of $H$ is the subgroup of  $\GL_2(\mathbb{Z}_7)$ that reduces to the subgroup of $\GL_2(\mathbb{F}_7)$ which under the projection map to $\mathrm{PGL}_2(\mathbb{F}_7)$ is the subgroup of order $24$ of $\mathrm{PSL}_2(\mathbb{F}_7)$ containing the normalizer of non-split Cartan subgroup defining $G_2$.
  \end{itemize}
We have the containments $G_2\subset G_1$, $G_2\subset H$, and $\det G_1= \widehat{\mathbb{Z}}^\times$ while $\det G_2= \det H$ is the subgroup of index $2$ of $\widehat{\mathbb{Z}}^\times$ consisting of elements whose component at $7$ reduces to a square. Thus the Shimura variety $\mathrm{Sh}_{G_1}=X$ is geometrically connected while $\mathrm{Sh}_{G_2}$, $\mathrm{Sh}_{H}$ have $2$ connected component over $\mathbb{Q}(\sqrt{-7})$. Since the element $\left(\begin{smallmatrix}  0 &-1 \\ 5 &0 \end{smallmatrix}\right)_5$ normalizes all three open compact subgroup and has determinant $5 \notin (\mathbb{F}_7^\times)^2$, it induces an involution $w$ over $\mathbb{Q}$ on all three Shimura varieties, and permutes the geometric connected components transitively. Putting $\Gamma_G= G_1\cap\mathrm{SL}_2(\mathbb{Q})=G_2 \cap \mathrm{SL}_2(\mathbb{Q})$ and $\Gamma_H= H\cap\mathrm{SL}_2(\mathbb{Q})$, we have a commuting diagram of complex curves with involution $w$:
\begin{equation}\label{Modular curve}
\xymatrix
{
	\mathrm{Sh}_{G_2}=\Gamma _G \setminus \mathbb{H} \amalg \Gamma _G\setminus \mathbb{H} \ar[d] \ar[r] &  \mathrm{Sh}_{G_1}=\Gamma _G \setminus \mathbb{H}\ar[d] \\
	\mathrm{Sh}_{H}=\Gamma _H \setminus \mathbb{H} \amalg \Gamma _H \setminus \mathbb{H}  \ar[r] & X(5b)
}
\end{equation}
where the vertical map is given by the quotient map on each component, while the horizontal map is the identity on the first component and $w$ on the second component.
The above diagram descends to $K$. The $\mathbb{Q}$-structure on $\mathrm{Sh}_{G_1}=X$ is determined by the subfield $\mathbb{Q}(x,\phi)$ inside its function field over $\mathbb{C}$. Since all arrows respect the $\mathbb{Q}$-structures, we see that there is an isomorphism of curves over $K$
\[d: Y\cong  Y^\sigma=Y\times_{K,\sigma}K \]
and a commutative diagram of curves over $K$
\begin{equation}
\xymatrix
{
	X \ar[r]^\pi \ar[d]^w & Y 
	\ar[d]^d \\
	X \ar[r]^{\pi^\sigma} & Y^\sigma 
}
\end{equation}
such that $d(x)=125/x$. Note that there is at most one $d$ with such property, and using Magma we compute that 
\[d(-\psi)= 
      \frac{P(x,\psi)}{(x^2+4x-1)(x^2+10x+5)^2}
\]
with
\\
$P(x,\psi)=4x\psi^6 + (-9x^2 -
        48x + 25)\psi^5 +
        (1/2(3\sqrt{-7} + 3)x^3 + (22\sqrt{-7} + 22)x^2 + 1/2(177\sqrt{-7} + 33)x)\psi^4 + (-x^4 - 24x^3 + 1/2(135\sqrt{-7} -
        453)x^2 + (336\sqrt{-7} - 536)x + 1/2(-375\sqrt{-7} + 375))\psi^3 + ((-3\sqrt{-7} - 3)x^4 + 1/2(-97\sqrt{-7} + 47)x^3 +
        (-252\sqrt{-7} + 894)x^2 + 1/2(-459\sqrt{-7} + 6621)x + (125\sqrt{-7} - 125))\psi^2 + ((-\sqrt{-7} - 1)x^5 + (-21\sqrt{-7}
        - 30)x^4 + (-99\sqrt{-7} - 363)x^3 + 1/2(727\sqrt{-7} - 1923)x^2 + (1512\sqrt{-7} +
        2208)x + 1/2(-1125\sqrt{-7} - 5625))\psi + ((-3\sqrt{-7} - 69)x^4 + (-180\sqrt{-7} - 1092)x^3 + (-2331\sqrt{-7} -
        4761)x^2 + (-6228\sqrt{-7} - 3444)x + (750\sqrt{-7} + 750)).$

\begin{lem}
Inside $H^0(X_K,\Omega^1)=H^0(X,\Omega^1)\otimes_\mathbb{Q}K$ we have
\[ \pi^*H^0(Y,\Omega^1)\cap w^*\pi^*H^0(Y,\Omega^1)=0\]
\end{lem}
\begin{proof}
  By looking at the pullbacks of differentials in the diagram (\ref{Modular curve}), we see that $V \cap w^* V$ is stable under the anemic Hecke algebra $\mathbb{T}$ (the algebra generated by Hecke operators at good primes), we see that $V\cap w^*V$ is $\mathbb{T}$-stable. Because the Hecke fields of $A_i$ are totally real quadratic fields, the $H^0(A_i,\Omega^1)\otimes K$ are exactly the irreducible $\mathbb{T}\otimes K$ submodules of $H^0(X_K,\Omega^1)$. Hence if $V\cap w^*V\neq 0$, it must be $2$-dimensional and hence $V=w^*V=V^\sigma$ must be $H^0(A_i,\Omega^1)\otimes K$ for some $i$. But a non-zero element of this intersection gives rise to a vector $v\in \pi_i$ which is unramified away from $5$ and $7$, fixed under the Iwahori open compact at $5$, the normalizer of the non-split Cartan open compact at $7$ n and also under an $S_4$ subgroup of $\mathrm{PSL}_2(\mathbb{F}_7)$. The last two condition however forces $v$ be invariant under $\GL_2(\mathbb{Z}_7)$, contradicting the fact that $\pi_{i,7}$ has conductor $7^2$ (since it appears in the new part of $X_0(245)$).
\end{proof}
The lemma implies
\[\pi^*H^0(Y,\Omega^1)\oplus w^*\pi^*H^0(Y,\Omega^1)= (H^0(A_1,\Omega^1)\oplus H^0(A_2,\Omega^1))\otimes K\]
\begin{lem}
  Let $D$ be a degree $1$ divisor on $X$. Then the map $\Pi_D:X\to J(Y)$ given by
  \[P \mapsto \pi^\sigma(P)-\pi^\sigma(wP) -(\pi^\sigma(D)-\pi^\sigma(wD)) \]
  factorizes through the composition of $AJ_D: X\to J(X)$ and a $\mathbb{Q}$-quotient of $J(X)$ isogenous to $A_2$, and in particular through a quotient that has $\mathbb{Q}$-rank $0$.
\end{lem}
\begin{proof}
Since $D$ maps to $0$ in $J(Y^\sigma)$, there is a factorization through the Abel-Jacobi map associated to $D$.
Let $I$ be the ideal of $\mathbb{T}$ which cuts out the Hecke field of $A_3$. The observation after the previous lemma shows that $I J(X)$ gets killed in $J(Y^\sigma)$.
On the other hand, the image of $(w+1)(P-D)=(wP-D)+(P-D)-(wD-D)$ in $J(Y^\sigma)$ is 
\begin{align*}
  &\quad \pi^\sigma(wP)-\pi^\sigma(w^2P) +  \pi^\sigma(P)-\pi^\sigma(wP) \\
  &- (\pi^\sigma(wD)-\pi^\sigma(w^2P)) -(\pi^\sigma(D)-\pi^\sigma(wD))=0 
\end{align*}
hence the map in consideration factors through $J(X)/((w+1)J(X)+IJ(X))$. This factor is defined over $\mathbb{Q}$ and is $\mathbb{Q}$-isogenous to $A_2$.
\end{proof}
A convenient choice for the base divisor $D$ is given below
\begin{prop}
  Let 
  \begin{align*}
    &X^3-7X^2+7X+7=(X-\phi_1)(X-\phi_2)(X-\phi_3)\\
    &X^2+22X+125=(X-x_1)(X-x_2)
  \end{align*}
  \begin{itemize}
    \item
    In terms of the coordinate $(x,\phi)$, the $6$ cusps of $X$ are given by $(0,\phi_i)$ and $(\infty,\phi_i)$
    \item 
     In terms of the coordinate $(x,\psi)$, the $2$ cusps of $Y$ are given by $(0,\infty)$ and $(\infty,\infty)$
\item 
Put $D=(0,\phi_1)+(0,\phi_2)+(0,\phi_3)-(x_1,3)-(x_2,3)$ is a rational divisor of degree $1$ on $X$, and
\[ \pi^\sigma(D)-\pi^\sigma(wD)=3((0,\infty)-(\infty,\infty))\]
is a torsion point of exact order $7$ in $J(Y^\sigma)$.
  \end{itemize}
\end{prop}
\begin{proof}
  The first two statements are clear: Note that in terms of the singular plane model with coordinates $(x,\phi)$, the points $(\infty,\phi_i)$ are singular points which are of the type $x^5=y^7$, and the singularity is resolved after $3$ blowups, and at each step there is a unique point in the pre-image of the singularity. Hence each $(x,\phi_i)$ actually gives exactly $1$ point in the smooth curve $X$ (alternatively, one could check that $X$ has exactly $6$ cusps, and we have written down at least $6$ of them). A similar analysis gives the statement for $Y$ and $Y^\sigma$. Because the involution $w$ on $X$ descends to the Atkin-Lehner involution on $X(5b)$, we see that $w$ switches the fibers of $X\to X(5b)$ above $x=0$ and $x=\infty$. On the other hand, one checks that $(x_i,3)$ are the only $\mathbb{Q}(\sqrt{-1})$-rational points with $x=x_i$, and hence $w$ must switch them. From this the equality in the last item follows.
  Finally, we have $7((0,\infty)-(\infty,\infty))=\mathrm{div}(x)$ is principal, and $(0,\infty)-(\infty,\infty)$ is not principal since $Y^\sigma$ has genus $2$.
\end{proof}
Let us now take $D$ as in the proposition and consider the map $\Pi=\Pi_D$ defined above. If $P$, $Q=P^\tau$ are a conjugate pair of quadratic points on $X$, $AJ_D(P+Q)$ is a rational point on $J(X)$ and hence must map to a $7$-torsion point in $J(Y^\sigma)$ under $\Pi$, since it factors through a quotient isogenous to $A_2$. On the other hand $wD$ is also rational and maps to $6((0,\infty)-(\infty,\infty))$, which has exact order $7$ in $J(Y^\sigma)$, thus we have 
\begin{align*}
  \pi^\sigma(P)+\pi^\sigma(Q) -\pi^\sigma(wP)-\pi^\sigma(wQ) \sim a((0,\infty)-(\infty,\infty))
 \end{align*}
for some $a \in \mathbb{Z}$. Thus $P$ and $Q$ maps to the same point in the Kummer surface $K(Y^\sigma)=J(Y^\sigma)/\pm$ under the map $\Pi_b$
\[P \mapsto  \pi^\sigma(P)-\pi^\sigma(wP)-b((0,\infty)-(\infty,\infty))\]
for some suitable integer $b$. 

By finding an explicit basis $\Omega_1$, $\Omega_2$ for the space of homolorphic differentials on $Y^\sigma$ using Magma, we compute in terms of $x$, $\psi$ a double cover map
\begin{align*}
  u:Y^\sigma \to \mathbb{P}^1
\end{align*}
and a rational function $v$ realizing $Y^\sigma$ as a hyperelliptic curve of genus 2 of the form
\begin{align*}
  v^2= \textrm{sextic in } u
\end{align*}
\begin{lem}\label{Birational image}
  The map $\Pi_b : X \to J(Y^\sigma)$ is birational onto its image.
\end{lem}
\begin{proof}
  Suppose the contrary, so we have infinitely many pairs $(P_i,Q_i)$ of distinct points which have the same image via $\Pi_b$. We have 
  \begin{align*}
    \pi^\sigma(P_i)+\pi^\sigma(wQ_i)\sim \pi^\sigma(wP_i)+\pi^\sigma(Q_i)
  \end{align*}
  Suppose first that for infinitely many $i$, this effective degree 2 divisor is not the hyper elliptic class in $Y^\sigma$. This forces the above linear equivalence to be an equality of divisors. Note that $\pi^\sigma$ has degree 3, so $\pi^\sigma\circ w\neq \pi^\sigma$. Hence for infinitely many $i$ we must have $ \pi^\sigma(P_i)= \pi^\sigma(Q_i)$, $ \pi^\sigma(wP_i)= \pi^\sigma(wQ_i)$.
  
  If the above case does not happen, $ \pi^\sigma(P_i)+\pi^\sigma(wQ_i)$ and $\pi^\sigma(wP_i)+\pi^\sigma(Q_i)$ must be the hyperelliptic class for infinitely many $i$, as the hyperelliptic class is the unique $g^1_2$ on a genus 2 curve. This forces $P_i\neq wQ_i=:Q_i'$ for all but finitely many $i$, and we have
  \begin{align*}
    u(\pi^\sigma(P_i))&=u(\pi^\sigma(Q_i'))\\
    u(\pi^\sigma(wP_i))&=u(\pi^\sigma(wQ_i'))
  \end{align*}
  Thus in either case we see that the map 
  \begin{align*}
    (u\circ\pi^\sigma,u\circ \pi^\sigma \circ w):X \to \mathbb{P}^1\times \mathbb{P}^1
\end{align*}
  is not generically injective. However a Magma computation shows that the pair $(u\circ\pi^\sigma,u\circ \pi^\sigma \circ w)$ generates the function field of $X$, a contradiction.
\end{proof}
Consequently, composing the above map with the quotient map to the Kummer surface, we get a map $X \to \mathcal{K}(Y^\sigma)$ which is either birational onto its image or factors through $X/w$, which then is birational onto its image (since $X\to  X/w$ is the only degree $2$  map from $X$ to any curve). The second case happens if and only if pairs of the form $(P,wP)$ have the same image, and this happens if and only if $b=0$.

We are therefore reduced to finding conjugate pairs of quadratic points $(P,Q=P^\tau)$ on $X$ which maps to the same point in the Kummer surface via one of the above maps (note that we only need to consider $b\in \{0,1,2,3\}$, by replacing $P$, $Q$ with $wP$, $wQ$ if needed).

Let us first study the case $b=0$. The same argument in the proof of lemma 
\ref{Birational image} implies that either $P$,$P^\tau$ or $P$, $wP^\tau$ have the same image under the map
\begin{align*}
    (u\circ\pi^\sigma,u\circ \pi^\sigma \circ w):X \to \mathbb{P}^2
\end{align*}
which is birational onto its image. One checks that $u\circ\pi^\sigma$ realizes $K(X)$ as a degree 6 extension of $K(u\circ \pi^\sigma \circ w)$ and vice versa, and hence the image of $X$ is an irreducible plane curve of degree 12. Using Magma, we computed this plane curve explicitly and determined all its singular points defined over a quadartic extension of $K$. Hence either $P\neq P^\tau$ and $P\neq wP^\tau$, their common image must be one of the above singular points of the image of $X$ in $\mathbb{P}^2$; or $P=P^\tau$ or $P=wP^\tau$. One checks that the elliptic curves corresponding to such $P$ must be either a $\mathbb{Q}$-curve or have CM.

Finally, we now turn to the case $b\neq 0$. We are looking for pairs $(P,Q)$ such that 
\begin{align*}
  \pi^\sigma(P)+\pi^\sigma(Q) -\pi^\sigma(wP)-\pi^\sigma(wQ) \sim 2b((0,\infty)-(\infty,\infty))
 \end{align*}
 and we know a priori that there are only finitely many such pairs $(P,Q=P^\tau)$ (note that this was not true when $b=0$). By enumerating such pairs $(P,Q)$ over some primes $p$ split in $K$ where the whole situation has good reduction, we found for some primes $p$ there were no pairs $(P,Q)\in X(\mathbb{F}_p)^2$ or conjugate pairs $(P,P^\tau)\in X(\mathbb{F}_{p^2})$ satisfying the above equation, and hence there are no conjugate pairs of quadratic points on $X$ of this type. For $b=2$, a similar enumeration for the prime $p=71$ shows that the pairs $(P,Q)$ mod $p$ satisfying the linear equivalence relation for $b=2$ are either the cusps or the mod $p$ reduction of the pair of conjugate points corresponding to the CM point $P=(125\sqrt{5} + 250, -\frac{1}{2}(\sqrt{5} - 1))$ in the coordinates $(x,\phi)$. Furthermore, one checks that these are only possible pairs in $\mathbb{Q}_{71^2}$ lifting the pairs mod $71$. This shows that in this case all pairs of conjugate quadratic points we look for gives rise to CM elliptic curves.
 
 Putting everything together, we found that all quadratic points on $X$ gives rise to modular $j$-invariants.
\bibliographystyle{amsalpha}
\bibliography{Test}

\providecommand{\bysame}{\leavevmode\hbox to3em{\hrulefill}\thinspace}
\providecommand{\MR}{\relax\ifhmode\unskip\space\fi MR }
\providecommand{\MRhref}[2]{%
  \href{http://www.ams.org/mathscinet-getitem?mr=#1}{#2}
}
\providecommand{\href}[2]{#2}
\begin{thebibliography}{BCDT01}

\bibitem[Abr96]{abramovich_linear_1996}
Dan Abramovich, \emph{A linear lower bound on the gonality of modular curves},
  International Mathematics Research Notices (1996), no.~20,
  1005{\textendash}1011.

\bibitem[AL78]{atkin_twists_1978}
A.~O.~L. Atkin and Wen Ching~Winnie Li, \emph{Twists of newforms and
  pseudo-eigenvalues of {$w$-operators}}, Inventiones Mathematicae \textbf{48}
  (1978), no.~3, 221{\textendash}243.

\bibitem[All13]{allen_modularity_2013}
Patrick~B. Allen, \emph{Modularity of nearly ordinary 2-adic residually
  dihedral {G}alois representations}, {arXiv} e-print 1301.1113, January 2013.

\bibitem[Bar13]{Bar}
Burcu Baran, \emph{An exceptional isomorphism between modular curves of level
  13 (formerly "modular curves of level 13"), preprint.},
  \newline\url{http://www-personal.umich.edu/~bubaran/}.

\bibitem[BCDT01]{breuil_modularity_2001}
Christophe Breuil, Brian Conrad, Fred Diamond, and Richard Taylor, \emph{On the
  modularity of elliptic curves over $\mathbb{Q}$: wild 3-adic exercises},
  Journal of the American Mathematical Society \textbf{14} (2001), no.~4,
  843{\textendash}939 (electronic).

\bibitem[BCP97]{bosma_magma_1997}
Wieb Bosma, John Cannon, and Catherine Playoust, \emph{The {M}agma algebra
  system. {I}. the user language}, Journal of Symbolic Computation \textbf{24}
  (1997), no.~3-4, 235{\textendash}265, Computational algebra and number theory
  (London, 1993).

\bibitem[BD12]{breuil_formes_2012}
Christophe Breuil and Fred Diamond, \emph{Formes modulaires de {H}ilbert modulo
  $p$ et valeurs d'extensions {G}aloisiennes}, {arXiv} e-print 1208.5367,
  August 2012.

\bibitem[BLGG10]{barnet-lamb_congruences_2010}
Thomas Barnet-Lamb, Toby Gee, and David Geraghty, \emph{Congruences between
  {H}ilbert modular forms: constructing ordinary lifts}, {arXiv} e-print
  1006.0466, June 2010.

\bibitem[BLGG12]{barnet-lamb_congruences_2012}
\bysame, \emph{Congruences betwen {H}ilbert modular forms: constructing
  ordinary lifts {II}}, {arXiv} e-print 1205.4491, May 2012.

\bibitem[BR93]{blasius_motives_1993}
Don Blasius and Jonathan~D. Rogawski, \emph{Motives for {H}ilbert modular
  forms}, Inventiones Mathematicae \textbf{114} (1993), no.~1,
  55{\textendash}87.

\bibitem[Car83]{carayol_sur_1983}
Henri Carayol, \emph{Sur les repr{\'e}sentations $l$-adiques attach{\'e}es aux
  formes modulaires de {H}ilbert}, Comptes Rendus des S{\'e}ances de
  {l'Acad{\'e}mie} des Sciences. S{\'e}rie I. Math{\'e}matique \textbf{296}
  (1983), no.~15, 629{\textendash}632.

\bibitem[CDT99]{conrad_modularity_1999}
Brian Conrad, Fred Diamond, and Richard Taylor, \emph{Modularity of certain
  potentially barsotti-tate {G}alois representations}, Journal of the American
  Mathematical Society \textbf{12} (1999), no.~2, 521{\textendash}567.

\bibitem[CP04]{cummins_congruence_2004}
C.~J. Cummins and S.~Pauli, \emph{Congruence subgroups of
  $\textrm{PSL}_2(\mathbb{Z})$}, Symmetry in physics, {CRM} Proc. Lecture
  Notes, vol.~34, Amer. Math. Soc., Providence, {RI}, 2004, pp.~23--29.

\bibitem[DDT94]{darmon_fermats_1994}
Henri Darmon, Fred Diamond, and Richard Taylor, \emph{{F}ermat's last theorem},
  Current developments in mathematics, 1995 (Cambridge, {MA)}, Int. Press,
  Cambridge, {MA}, 1994, p.~1{\textendash}154.

\bibitem[DF11]{dieulefait_fermat-type_2011}
Luis Dieulefait and Nuno Freitas, \emph{{F}ermat-type equations of signature
  $(13,13,p)$ via {H}ilbert cuspforms}, {arXiv} e-print 1112.4521, December
  2011.

\bibitem[Dia97]{diamond_taylor-wiles_1997}
Fred Diamond, \emph{The {T}aylor-{W}iles construction and multiplicity one},
  Inventiones Mathematicae \textbf{128} (1997), no.~2, 379{\textendash}391.

\bibitem[DS05]{diamond_first_2005}
Fred Diamond and Jerry Shurman, \emph{A first course in modular forms},
  Graduate Texts in Mathematics, vol. 228, Springer-Verlag, New York, 2005.

\bibitem[Edi96]{edixhoven_result_1996}
Bas Edixhoven, \emph{On a result of {I}min {C}hen}, {arXiv:alg-geom/9604008}
  (1996).

\bibitem[Eke90]{ekedahl_effective_1990}
Torsten Ekedahl, \emph{An effective version of {H}ilbert's irreducibility
  theorem}, S{\'e}minaire de Th{\'e}orie des Nombres, Paris
  1988{\textendash}1989, Progr. Math., vol.~91, Birkh{\"a}user Boston, Boston,
  {MA}, 1990, p.~241{\textendash}249.

\bibitem[Elk98]{elkies_elliptic_1998}
Noam~D. Elkies, \emph{Elliptic and modular curves over finite fields and
  related computational issues}, Computational perspectives on number theory
  (Chicago, {IL}, 1995), {AMS/IP} Stud. Adv. Math., vol.~7, Amer. Math. Soc.,
  Providence, {RI}, 1998, p.~21{\textendash}76.

\bibitem[Elk99]{elkies_klein_1999}
\bysame, \emph{The {K}lein quartic in number theory}, The eightfold way, Math.
  Sci. Res. Inst. Publ., vol.~35, Cambridge Univ. Press, Cambridge, 1999,
  p.~51{\textendash}101.

\bibitem[FLHS]{FLS}
Nuno Freitas, Bao Le~Hung, and Samir Siksek, \emph{Elliptic curves over real
  quadratic fields are modular, in preparation}.

\bibitem[FS13]{freitas_modularity_2013}
Nuno Freitas and Samir Siksek, \emph{Modularity and the {F}ermat equation over
  totally real fields}, {arXiv} e-print 1307.3162, July 2013.

\bibitem[HLTT13]{HLTT}
Michael Harris, Kai-Wen Lan, Richard Taylor, and Jack Thorne, \emph{On the
  rigid cohomology of certain {S}himura varieties, preprint.},
  \newline\url{http://www.math.ias.edu/~rtaylor/}.

\bibitem[HS91]{harris_bielliptic_1991}
Joe Harris and Joe Silverman, \emph{Bielliptic curves and symmetric products},
  Proceedings of the American Mathematical Society \textbf{112} (1991), no.~2,
  347{\textendash}356.

\bibitem[JM04]{jarvis_fermat_2004}
Frazer Jarvis and Paul Meekin, \emph{The {F}ermat equation over
  $\mathbb{Q}(\sqrt{2})$}, Journal of Number Theory \textbf{109} (2004), no.~1,
  182{\textendash}196.

\bibitem[JM08]{jarvis_modularity_2008}
Frazer Jarvis and Jayanta Manoharmayum, \emph{On the modularity of
  supersingular elliptic curves over certain totally real number fields},
  Journal of Number Theory \textbf{128} (2008), no.~3, 589{\textendash}618.

\bibitem[Kis09]{kisin_moduli_2009}
Mark Kisin, \emph{Moduli of finite flat group schemes, and modularity}, Annals
  of Mathematics. Second Series \textbf{170} (2009), no.~3,
  1085{\textendash}1180.

\bibitem[KW09a]{khare_serres_2009-1}
Chandrashekhar Khare and Jean-Pierre Wintenberger, \emph{Serre's modularity
  conjecture. i}, Inventiones Mathematicae \textbf{178} (2009), no.~3,
  485{\textendash}504.

\bibitem[KW09b]{khare_serres_2009}
\bysame, \emph{Serre's modularity conjecture. {II}}, Inventiones Mathematicae
  \textbf{178} (2009), no.~3, 505{\textendash}586.

\bibitem[Lan80]{langlands_base_1980}
Robert~P. Langlands, \emph{Base change for $\textrm{GL}(2)$}, Annals of
  Mathematics Studies, vol.~96, Princeton University Press, Princeton, {N.J.},
  1980.

\bibitem[LW12]{loeffler_computation_2012}
David Loeffler and Jared Weinstein, \emph{On the computation of local
  components of a newform}, Mathematics of Computation \textbf{81} (2012),
  no.~278, 1179{\textendash}1200.

\bibitem[Man01]{manoharmayum_modularity_2001}
J.~Manoharmayum, \emph{On the modularity of certain
  $\textrm{GL}_2(\mathbb{F}_7)$ {G}alois representations}, Mathematical
  Research Letters \textbf{8} (2001), no.~5-6, 703{\textendash}712.

\bibitem[OR10]{oyono_rationality_2010}
Roger Oyono and Christophe Ritzenthaler, \emph{On rationality of the
  intersection points of a line with a plane quartic}, Arithmetic of finite
  fields, Lecture Notes in Comput. Sci., vol. 6087, Springer, Berlin, 2010,
  p.~224{\textendash}237.

\bibitem[PS83]{piatetski-shapiro_complex_1983}
Ilya Piatetski-Shapiro, \emph{Complex representations of $\mathrm{GL}_2(k)$ for
  finite fields $k$.}, Contemporary Mathematics, vol.~16, American Mathematical
  Society, Providence, {R.I.}, 1983.

\bibitem[Rib04]{ribet_abelian_2004}
Kenneth~A. Ribet, \emph{Abelian varieties over $\mathbb{Q}$ and modular forms},
  Modular curves and abelian varieties, Progr. Math., vol. 224, Birkh{\"a}user,
  Basel, 2004, p.~241{\textendash}261.

\bibitem[Sch13]{scholze_torsion_2013}
Peter Scholze, \emph{On torsion in the cohomology of locally symmetric
  varieties}, {arXiv} e-print 1306.2070, June 2013.

\bibitem[Sil94]{silverman_advanced_1994}
Joseph~H. Silverman, \emph{Advanced topics in the arithmetic of elliptic
  curves}, Springer, November 1994.

\bibitem[ST68]{serre_good_1968}
Jean-Pierre Serre and John Tate, \emph{Good reduction of abelian varieties},
  Annals of Mathematics. Second Series \textbf{88} (1968), 492{\textendash}517.

\bibitem[Ste12]{mfd}
William Stein, \emph{The {M}odular {F}orms {D}atabase},
  \newline\url{http://wstein.org/Tables}.

\bibitem[Sti93]{stichtenoth_algebraic_1993}
Henning Stichtenoth, \emph{Algebraic function fields and codes}, Universitext,
  Springer-Verlag, Berlin, 1993.

\bibitem[SW99]{skinner_residually_1999}
C.~M. Skinner and A.~J. Wiles, \emph{Residually reducible representations and
  modular forms}, Institut des Hautes {\'E}tudes Scientifiques. Publications
  Math{\'e}matiques (1999), no.~89, 5{\textendash}126 (2000).

\bibitem[SW01]{skinner_nearly_2001}
C.~M. Skinner and Andrew~J. Wiles, \emph{Nearly ordinary deformations of
  irreducible residual representations}, Annales de la Facult{\'e} des Sciences
  de Toulouse. Math{\'e}matiques. S{\'e}rie 6 \textbf{10} (2001), no.~1,
  185{\textendash}215.

\bibitem[Tay89]{taylor_galois_1989}
Richard Taylor, \emph{On {G}alois representations associated to {H}ilbert
  modular forms}, Inventiones Mathematicae \textbf{98} (1989), no.~2,
  265{\textendash}280.

\bibitem[Tay95]{taylor_galois_1995}
\bysame, \emph{On {G}alois representations associated to {H}ilbert modular
  forms. {II}}, Elliptic curves, modular forms, \& {F}ermat's last theorem
  (Hong Kong, 1993), Ser. Number Theory, I, Int. Press, Cambridge, {MA}, 1995,
  p.~185{\textendash}191.

\bibitem[Tun81]{tunnell_artins_1981}
Jerrold Tunnell, \emph{Artin's conjecture for representations of octahedral
  type}, American Mathematical Society. Bulletin. New Series \textbf{5} (1981),
  no.~2, 173{\textendash}175.

\bibitem[TW95]{taylor_ring-theoretic_1995}
Richard Taylor and Andrew Wiles, \emph{Ring-theoretic properties of certain
  {H}ecke algebras}, Annals of Mathematics. Second Series \textbf{141} (1995),
  no.~3, 553{\textendash}572.

\bibitem[Wil88]{wiles_ordinary_1988}
A.~Wiles, \emph{On ordinary $\lambda$-adic representations associated to
  modular forms}, Inventiones Mathematicae \textbf{94} (1988), no.~3,
  529{\textendash}573.

\bibitem[Wil95]{wiles_modular_1995}
Andrew Wiles, \emph{Modular elliptic curves and {F}ermat's last theorem},
  Annals of Mathematics. Second Series \textbf{141} (1995), no.~3,
  443{\textendash}551.

\end{thebibliography}
E-mail address: \texttt{lhvietbao@googlemail.com}\\
Dept. of Mathematics, Harvard University, Cambridge, MA, USA

\end{document}